\documentclass[final,3p,times]{elsarticle}
\usepackage{amssymb}
\usepackage{amsmath}
\usepackage{amsthm}

\DeclareMathAlphabet{\mathcal}{OMS}{cmsy}{m}{n}

\newcommand{\bbC}{\mathbb{C}}
\newcommand{\bbF}{\mathbb{F}}
\newcommand{\bbR}{\mathbb{R}}

\newcommand{\bfA}{\mathbf{A}}
\newcommand{\bfG}{\mathbf{G}}
\newcommand{\bfI}{\mathbf{I}}
\newcommand{\bfx}{\mathbf{x}}
\newcommand{\bfX}{\mathbf{X}}
\newcommand{\bfy}{\mathbf{y}}

\newcommand{\bfalpha}{\boldsymbol\alpha}
\newcommand{\bflambda}{\boldsymbol\lambda}
\newcommand{\bfmu}{\boldsymbol\mu}
\newcommand{\bfphi}{\boldsymbol\varphi}
\newcommand{\bfpsi}{\boldsymbol\psi}
\newcommand{\bfPhi}{\boldsymbol\Phi}
\newcommand{\bfPsi}{\boldsymbol\Psi}

\newcommand{\calJ}{\mathcal{J}}

\newcommand{\Tr}{\mathrm{Tr}}

\newcommand{\abs}[1]{|{#1}|}

\newcommand{\bigparen}[1]{\bigl({#1}\bigr)}

\newcommand{\Biggparen}[1]{\Biggl({#1}\Biggr)}

\newcommand{\set}[1]{\{{#1}\}}
\newcommand{\bigset}[1]{\bigl\{{#1}\bigr\}}
\newcommand{\Bigset}[1]{\Bigl\{{#1}\Bigr\}}

\newcommand{\Biggset}[1]{\Biggl\{{#1}\Biggr\}}

\newcommand{\norm}[1]{\|{#1}\|}

\newcommand{\ip}[2]{\langle{#1},{#2}\rangle}

\theoremstyle{plain}
\newtheorem{theorem}{Theorem}
\newtheorem{lemma}{Lemma}

\theoremstyle{definition}
\newtheorem{definition}{Definition}
\newtheorem{example}{Example}

\begin{document}
\begin{frontmatter}
\title{A generalized Schur-Horn theorem and optimal frame completions}

\author[AFIT]{Matthew Fickus}
\ead{Matthew.Fickus@gmail.com}
\author[Bowdoin]{Justin Marks}
\author[AFIT]{Miriam J.~Poteet}

\address[AFIT]{Department of Mathematics and Statistics, Air Force Institute of Technology, Wright-Patterson Air Force Base, OH 45433, USA}
\address[Bowdoin]{Department of Mathematics, Bowdoin College, Brunswick, ME 04011, USA}

\begin{abstract}
The Schur-Horn theorem is a classical result in matrix analysis which characterizes the existence of positive semidefinite matrices with a given diagonal and spectrum.
In recent years, this theorem has been used to characterize the existence of finite frames whose elements have given lengths and whose frame operator has a given spectrum.
We provide a new generalization of the Schur-Horn theorem which characterizes the spectra of all possible finite frame completions.
That is, we characterize the spectra of the frame operators of the finite frames obtained by adding new vectors of given lengths to an existing frame.
We then exploit this characterization to give a new and simple algorithm for computing the optimal such completion.
\end{abstract}

\end{frontmatter}

\section{Introduction}
The Schur-Horn theorem~\cite{Horn54,Schur23} is a classical result in matrix analysis which characterizes the existence of positive-semidefinite  matrices with a given diagonal and spectrum.
To be precise, let $\bbF$ be either the real field $\bbR$ or the complex field $\bbC$, and let $\set{\lambda_n}_{n=1}^{N}$ and $\set{\mu_n}_{n=1}^{N}$ be any nonincreasing sequences of nonnegative real scalars.
The Schur-Horn theorem states that there exists a positive semidefinite matrix $\bfG\in\bbF^{N\times N}$ with eigenvalues $\set{\lambda_n}_{n=1}^{N}$ and with $\bfG(n,n)=\mu_n$ for all $n=1,\dotsc,N$ if and only if $\set{\lambda_n}_{n=1}^{N}$ \textit{majorizes} $\set{\mu_n}_{n=1}^{N}$, that is, precisely when
\begin{equation}
\label{equation.classical Schur-Horn}
\sum_{n=1}^{N}\mu_n=\sum_{n=1}^{N}\lambda_n,
\qquad
\sum_{n=1}^{j}\mu_n\leq\sum_{n=1}^{j}\lambda_n,\quad \forall j=1,\dotsc,N,
\end{equation}
denoted $\set{\mu_n}_{n=1}^{N}\preceq\set{\lambda_n}_{n=1}^{N}$.
The first part of~\eqref{equation.classical Schur-Horn} is simply a trace condition: the sum of the diagonal entries of $\bfG$ must equal the sum of its eigenvalues.
The second part of~\eqref{equation.classical Schur-Horn} is less intuitive.
To understand it better, it helps to have some basic concepts from finite frame theory.

For any finite sequence of vectors \smash{$\set{\bfphi_n}_{n=1}^{N}$} in $\bbF^M$, the corresponding \textit{synthesis operator} is the $M\times N$ matrix whose $n$th column is $\bfphi_n$, namely $\bfPhi:\bbF^N\rightarrow\bbF^M$, $\bfPhi\bfy:=\sum_{n=1}^{N}\bfy(n)\bfphi_n$.
Its $N\times M$ adjoint is the \textit{analysis operator} $\bfPhi^*:\bbF^M\rightarrow\bbF^N$, $(\bfPhi^*\bfx)(n):=\ip{\bfphi_n}{\bfx}$.
The vectors $\set{\bfphi_n}_{n=1}^{N}$ are a \textit{finite frame} for $\bbF^M$ if they span $\bbF^M$, which is equivalent to having their $M\times M$ \textit{frame operator} \smash{$\bfPhi\bfPhi^*=\sum_{n=1}^{N}\bfphi_n^{}\bfphi_n^*$} be invertible.
Here, $\bfphi_n^*$ is $1\times M$ adjoint of the $M\times 1$ column vector $\bfphi_n$, namely the linear operator $\bfphi_n^*\bfx=\ip{\bfphi_n}{\bfx}$.
The least and greatest eigenvalues $\alpha$ and $\beta$ of $\bfPhi\bfPhi^*$ are called the \textit{lower} and \textit{upper frame bounds} of $\set{\bfphi_n}_{n=1}^{N}$, and their ratio $\beta/\alpha$ is the \textit{condition number} of $\bfPhi\bfPhi^*$.
Inspired by applications involving additive noise, finite frame theorists often seek frames that are as well-conditioned as possible, the ideal case being \textit{tight frames} in which $\bfPhi\bfPhi^*=\alpha\bfI$ for some $\alpha>0$.
They also care about the lengths of the frame vectors, often requiring that $\norm{\bfphi_n}^2=\mu_n$ for some prescribed sequence $\set{\mu_n}_{n=1}^{N}$.
These lengths weight the summands of the linear-least-squares objective function \smash{$\norm{\bfPhi^*\bfx-\bfy}^2=\sum_{n=1}^{N}\abs{\ip{\bfphi_n}{\bfx}-\bfy(n)}^2$}, and adjusting them is closely related to the linear-algebraic concept of \textit{preconditioning}.
That is, we often want to control both the spectrum of the frame operator as well as the lengths of the frame vectors.
For example, much attention has been paid to finite tight frames whose vectors are unit norm~\cite{BenedettoF03,CasazzaK03,GoyalKK01,GoyalVT98}.

In this context, the reason we care about the Schur-Horn theorem is that it provides a simple characterization of when there exists a finite frame whose frame operator has a given spectrum and whose frame vectors have given lengths.
To elaborate, the earliest reference which briefly mentions the Schur-Horn theorem in the context of finite frames seems to be~\cite{TroppDHS05},
which stems from even earlier, closely related work on synchronous CMDA systems~\cite{ViswanathA99,ViswanathA02}.
An in-depth analysis of the connection between frame theory and the Schur-Horn theorem is given in~\cite{AntezanaMRS07}.
There as here, the main idea is to apply the Schur-Horn theorem to the \textit{Gram matrix} of a given sequence of vectors $\set{\bfphi_n}_{n=1}^{N}$, namely the $N\times N$ matrix $\bfPhi^*\bfPhi$ whose $(n,n')$th entry is $(\bfPhi^*\bfPhi)(n,n')=\ip{\bfphi_n}{\bfphi_{n'}}$.
Indeed, suppose there exists $\set{\bfphi_n}_{n=1}^{N}$ in $\bbF^M$ whose frame operator $\bfPhi\bfPhi^*$ has spectrum $\set{\lambda_m}_{m=1}^{M}$ and whose frame vectors have squared-norms $\norm{\bfphi_n}^2=\mu_n$ for all $n=1,\dotsc,N$.
The diagonal entries of $\bfPhi^*\bfPhi$ are $\set{(\bfPhi^*\bfPhi)(n,n)}_{n=1}^{N}=\set{\norm{\bfphi_n}^2}_{n=1}^N=\set{\mu_n}_{n=1}^{N}$ which, by reordering the frame vectors if necessary, we can assume are nonincreasing.
Meanwhile, the spectra of the Gram matrix $\bfPhi^*\bfPhi$ and the frame operator $\bfPhi\bfPhi^*$ are zero-padded versions of each other.
Since adjoining vectors of squared-length $\mu_n=0$ to a sequence $\set{\bfphi_n}_{n=1}^{N}$ does not change its $M\times M$ frame operator $\bfPhi\bfPhi^*$ we further assume without loss of generality that $M\leq N$, implying that the spectrum of $\bfPhi^*\bfPhi$ is $\set{\lambda_m}_{m=1}^{M}$ appended with $N-M$ zeros.
Applying the Schur-Horn theorem to $\bfPhi^*\bfPhi$ then implies that $\set{\lambda_m}_{m=1}^{M}\cup\set{0}_{m=M+1}^{N}$ necessarily majorizes $\set{\mu_n}_{n=1}^{N}$, with~\eqref{equation.classical Schur-Horn} reducing to
\begin{equation}
\label{equation.frame theory Schur-Horn}
\sum_{n=1}^{N}\mu_n=\sum_{m=1}^{M}\lambda_m,
\qquad
\sum_{n=1}^{j}\mu_n\leq\sum_{m=1}^{j}\lambda_m,\quad \forall j=1,\dotsc,M.
\end{equation}
Conversely, for any $M\leq N$ and any nonnegative nonincreasing sequences $\set{\lambda_m}_{m=1}^{M}$ and $\set{\mu_n}_{n=1}^{N}$ that satisfy~\eqref{equation.frame theory Schur-Horn}, the Schur-Horn theorem also implies that there exists a positive semidefinite matrix with spectrum $\set{\lambda_m}_{m=1}^{M}\cup\set{0}_{m=M+1}^{N}$ and with diagonal entries $\set{\mu_n}_{n=1}^{N}$.
Since the rank of $\bfG$ is at most $M$, taking the singular value decomposition of $\bfG$ allows it to be written as $\bfG=\bfPhi^*\bfPhi$ where $\bfPhi\in\bbF^{M\times N}$ has singular values \smash{$\set{\lambda_m^{1/2}}_{m=1}^{M}$}.
Letting $\set{\bfphi_n}_{n=1}^{N}$ denote the columns of this matrix $\bfPhi$, we see that there exists $N$ vectors in $\bbF^M$ whose frame operator $\bfPhi\bfPhi^*$ has spectrum $\set{\lambda_m}_{m=1}^{M}$ and where $\norm{\bfphi_n}^2=\mu_n$ for all $n=1,\dotsc,N$.

In summary, for any $M\leq N$ and any nonnegative nonincreasing sequences $\set{\lambda_m}_{m=1}^{M}$ and $\set{\mu_n}_{n=1}^{N}$, the Schur-Horn theorem gives that there exists $\set{\bfphi_n}_{n=1}^{N}$ in $\bbF^M$ where $\bfPhi\bfPhi^*$ has spectrum $\set{\lambda_m}_{m=1}^{M}$ and where $\norm{\bfphi_n}^2=\mu_n$ for all $n$ if and only if~\eqref{equation.frame theory Schur-Horn} holds.
Note that in the $M=N$ case, this statement reduces the classical Schur-Horn theorem and as such, is an equivalent formulation of it.
This equivalence allows the Schur-Horn and finite frame theory communities to contribute to each other.
For example, the Schur-Horn theorem gives frame theorists another reason why there exists a unit norm tight frame of $N$ vectors in $\bbF^M$ for any $M\leq N$: the sequence $\set{\lambda_m}_{m=1}^{M}=\set{\frac NM}_{m=1}^{M}\cup\set{0}_{m=M+1}^{N}$ majorizes the constant sequence $\set{\mu_n}_{n=1}^{N}=\set{1}_{n=1}^{N}$.
In the other direction, techniques originally developed to characterize the existence of finite frames, such as the Givens-rotation-based constructions of~\cite{CasazzaL02} and the optimization-based methods of~\cite{CasazzaFKLT06}, are meaningful contributions to the existing ``proof of Schur-Horn" literature~\cite{ChanL83,Chu95,DavisH00,DhillonHST05,HornJ85,LeiteRT99}.

Frame theory also provides the Schur-Horn community with a geometric interpretation of the inequalities in~\eqref{equation.classical Schur-Horn} and~\eqref{equation.frame theory Schur-Horn}.
To be precise, for any vectors $\set{\bfphi_n}_{n=1}^{N}$ in $\bbF^M$ and any $j=1,\dotsc,M$, the quantity \smash{$\sum_{n=1}^{j}\mu_n$} is the trace of the \textit{$j$th partial frame operator} $\bfPhi_j^{}\bfPhi_j^*$, where $\bfPhi_j$ denotes the synthesis operator of \smash{$\set{\bfphi_n}_{n=1}^{j}$}:
\begin{equation}
\label{equation.trace condition}
\sum_{n=1}^{j}\mu_n
=\sum_{n=1}^{j}\norm{\bfphi_n}^2
=\sum_{n=1}^{j}\bfphi_n^*\bfphi_n^{}
=\sum_{n=1}^{j}\Tr(\bfphi_n^*\bfphi_n^{})
=\sum_{n=1}^{j}\Tr(\bfphi_n^{}\bfphi_n^*)
=\Tr\Biggparen{\sum_{n=1}^{j}\bfphi_n^{}\bfphi_n^*}
=\Tr(\bfPhi_j^{}\bfPhi_j^*).
\end{equation}
Here, the $n$th summand of \smash{$\bfPhi_j^{}\bfPhi_j^*=\sum_{n=1}^{j}\bfphi_n^{}\bfphi_n^*$} is the orthogonal projection operator onto the line spanned by $\bfphi_n$, scaled by a factor of $\norm{\bfphi_n}^2=\mu_n$.
Since the vectors $\set{\bfphi_n}_{n=1}^{j}$ span at most a $j$-dimensional space, all but $j$ of the eigenvalues of $\bfPhi_j^{}\bfPhi_j^*$ are zero.
As such, \smash{$\Tr(\bfPhi_j^{}\bfPhi_j^*)=\sum_{n=1}^{j}\mu_n$} is the sum of the $j$ largest eigenvalues of $\bfPhi_j^{}\bfPhi_j^*$.
Moreover, as we add the remaining scaled-projections \smash{$\set{\bfphi_n^{}\bfphi_n^*}_{n=j+1}^{N}$} to $\bfPhi_j^{}\bfPhi_j^*$ in order to form $\bfPhi\bfPhi^*$, these $j$ largest eigenvalues will only grow larger, leading to the $j$th inequality in~\eqref{equation.frame theory Schur-Horn};
formally this follows from the rules of \textit{eigenvalue interlacing}, as detailed in the next section.

The remarkable fact about the Schur-Horn theorem is that these relatively easy-to-derive necessary conditions~\eqref{equation.frame theory Schur-Horn} are also sufficient.
Many of the traditional proofs of the sufficiency of~\eqref{equation.frame theory Schur-Horn} involve explicit constructions.
And, of these, only the recently-introduced \textit{eigenstep}-based construction method of~\cite{CahillFMPS13,FickusMPS13} is truly general in the sense that for a given $\set{\lambda_m}_{m=1}^{M}$ and $\set{\mu_n}_{n=1}^{N}$ it can construct every finite frame of the corresponding type.
In this paper, we further exploit the power of the eigensteps method, generalizing the Schur-Horn theorem so that it applies to another type of problem in finite frame theory.

In particular, in this paper we derive a generalized Schur-Horn theorem that addresses the \textit{frame completion problem}: given an initial frame, which new vectors should be appended to it in order to make it a better frame?
More precisely, given an initial sequence of vectors whose frame operator is some $M\times M$ positive semidefinite matrix $\bfA$, how should we choose
$\set{\bfphi_n}_{n=1}^{N}$ so that the frame operator of the entire collection, namely $\bfA+\sum_{n=1}^{N}\bfphi_n^{}\bfphi_n^*$, is optimally well-conditioned?
Finite frames have been used to model sensor networks~\cite{RanieriCV14}; from that perspective, the completion problem asks what sensors should we add to an existing sensor network so that the new network is as robust as possible against measurement error and noise.

The frame completion problem was first considered in~\cite{FengLY06}.
There, the authors characterized the smallest number $N$ of new vectors that permits $\bfA+\sum_{n=1}^{N}\bfphi_n^{}\bfphi_n^*$ to be tight, provided $\set{\bfphi_n}_{n=1}^{N}$ can be arbitrarily chosen.
They also gave a lower bound on the smallest such $N$ in the case where each $\bfphi_n$ is required to have unit norm.
Shortly thereafter in~\cite{MasseyR08}, the classical Schur-Horn theorem was used to completely characterize the smallest such $N$ in the case where the squared-norms of $\set{\bfphi_n}_{n=1}^{N}$ are some arbitrary nonnegative nonincreasing values $\set{\mu_n}_{n=1}^{N}$.
This prior work naturally leads to several new problems, a couple of which we solve in this paper.
It helps here to introduce some terminology:
\begin{definition}
\label{definition.completion}
Given nonnegative nonincreasing sequences \smash{$\bfalpha=\set{\alpha_m}_{m=1}^{M}$} and \smash{$\bfmu=\set{\mu_n}_{n=1}^{N}$}, we say a nonnegative nonincreasing sequence $\bflambda=\set{\lambda_m}_{m=1}^{M}$ is an \textit{$(\bfalpha,\bfmu)$-completion} if $\bflambda$ is the spectrum of some operator of the form \smash{$\bfA+\sum_{n=1}^{N}\bfphi_n^{}\bfphi_n^*$} where $\bfA$ is a self-adjoint matrix with spectrum $\bfalpha$ and where $\norm{\bfphi_n}^2=\mu_n$ for all $n=1,\dotsc,N$.
\end{definition}
Our first main result characterizes all $(\bfalpha,\bfmu)$-completions via a generalized Schur-Horn theorem.
\begin{theorem}
\label{theorem.main result 1}
For any nonnegative nonincreasing sequences $\bfalpha=\set{\alpha_m}_{m=1}^{M}$ and $\bfmu=\set{\mu_n}_{n=1}^{N}$, a nonnegative nonincreasing sequence $\set{\lambda_m}_{m=1}^{M}$ is an $(\bfalpha,\bfmu)$-completion if and only if $\lambda_m\geq\alpha_m$ for all $m$ and:
\begin{equation}
\label{equation.completion Schur-Horn}
\sum_{m=1}^{M}(\lambda_m-\alpha_m)=\sum_{n=1}^{N}\mu_n,
\qquad
\sum_{m=j}^{M}(\lambda_m-\alpha_{m-j+1})^+\leq \sum_{n=j}^{N}\mu_n,\quad\forall j=1,\dotsc,M.
\end{equation}
\end{theorem}
Here, $x^+:=\max\set{0,x}$ denotes the positive part of a real scalar $x$.
Moreover, note here we have made no assumption that $M\leq N$; in the case where $N<j\leq M$, the sums on the right-hand side of~\eqref{equation.completion Schur-Horn} are taken over an empty set of indices and, like all other empty sums in this paper, are defined by convention to be zero.
This convention is consistent with defining $\mu_n:=0$ for all $n>N$, though we choose not to interpret this particular result in this way in order to facilitate its proof.
Note that under this convention, \eqref{equation.completion Schur-Horn} holds for a given $j$ such that $N<j\leq M$ if and only if $\lambda_m\leq\alpha_{m-j+1}$ for all $m=j,\dotsc,M$.
We also remark on an aspect of Theorem~\ref{theorem.main result 1} that one of the anonymous reviewers kindly pointed out: 
the condition that $\alpha_m\leq\lambda_m$ for all $m$ is superfluous, being implied by~\eqref{equation.completion Schur-Horn}.
Indeed, combining the equality condition of~\eqref{equation.completion Schur-Horn} with the inequality condition when $j=1$ gives 
$\sum_{m=1}^{M}(\lambda_m-\alpha_m)^+\leq \sum_{n=1}^{N}\mu_n=\sum_{m=1}^{M}(\lambda_m-\alpha_m)$.
Since $\lambda_m-\alpha_m\leq (\lambda_m-\alpha_m)^+$ for all $m$, this is only possible if $\lambda_m-\alpha_m=(\lambda_m-\alpha_m)^+$ for all $m$, that is, when $\lambda_m\geq\alpha_m$ for all $m$.
Nevertheless, we explicitly retain this condition in the statement of Theorem~\ref{theorem.main result 1}, as it facilitates the intuition and proof techniques we develop below.

The traditional Schur-Horn theorem is a special case of Theorem~\ref{theorem.main result 1} when $\alpha_m=0$ for all $m$.
Indeed, a nonnegative nonincreasing sequence $\set{\lambda_m}_{m=1}^{M}$ is a $(\mathbf{0},\bfmu)$-completion precisely when it is the spectrum of some frame operator $\sum_{n=1}^{N}\bfphi_n^{}\bfphi_n^*$ where $\norm{\bfphi_n}^2=\mu_n$ for all $n$.
Meanwhile, in this same case, the conditions of~\eqref{equation.completion Schur-Horn} reduce to
\begin{equation*}
\sum_{m=1}^{M}\lambda_m=\sum_{n=1}^{N}\mu_n,
\qquad
\sum_{m=j}^{M}\lambda_m\leq\sum_{n=j}^{N}\mu_n,\quad\forall j=1,\dotsc,M.
\end{equation*}
Subtracting these inequalities from the equality, we see these conditions are a restatement of~\eqref{equation.frame theory Schur-Horn}.

The next section is devoted to the proof of Theorem~\ref{theorem.main result 1}.
The proof of the necessity of \eqref{equation.completion Schur-Horn} follows quickly from the classical principle of eigenvalue interlacing.
On the other hand, the proof of its sufficiency relies on a nontrivial generalization of the eigensteps method of~\cite{CahillFMPS13,FickusMPS13}.

In Section~3, we then use this new characterization of all $(\bfalpha,\bfmu)$-completions to find the optimal such completion;
this problem was first posed in~\cite{MasseyRS13}, a generalization of one given in~\cite{FickusMP11}.
In particular, in contrast to~\cite{FengLY06,MasseyR08} which characterize what $\bfalpha$'s and $\bfmu$'s permit a tight (constant) completion $\set{\lambda_m}_{m=1}^{M}$, we take an arbitrary $\bfalpha$ and $\bfmu$ and compute the tightest $(\bfalpha,\bfmu)$-completion.
Here, one naturally asks how we should quantify tightness.
Should we make the condition number $\lambda_1/\lambda_M$ as small as possible?
If so, how is this related to making $\lambda_M$ and $\lambda_1$ as large and small as possible, respectively?
Alternatively, should we maybe minimize the mean squared reconstruction error $\sum_{n=1}^{N}1/\lambda_n$ of~\cite{GoyalVT98} or the frame potential $\sum_{n=1}^{N}\lambda_n^2$ of~\cite{BenedettoF03}?
Surprisingly, there exists a single completion that does all these things and more.

The key idea, as similarly exploited in~\cite{MasseyRS13,MasseyRS14a,MasseyRS14b,ViswanathA02}, is that majorization itself yields a partial order on the set of all $(\bfalpha,\bfmu)$-completions.
To be precise, note that by the equality condition of Theorem~\ref{theorem.main result 1}, any two such completions $\set{\beta_m}_{m=1}^{M}$ and $\set{\lambda_m}_{m=1}^{M}$ have the same sum, namely \smash{$\sum_{m=1}^{M}\beta_m=\sum_{m=1}^{M}\alpha_m+\sum_{n=1}^{N}\mu_n=\sum_{m=1}^{M}\lambda_m$}.
Thus, $\set{\beta_m}_{m=1}^{N}\preceq\set{\lambda_m}_{m=1}^{M}$ when
\begin{equation}
\label{equation.majorization from bottom}
\sum_{m=1}^{j}\beta_m\leq\sum_{m=1}^{j}\lambda_m,\quad\forall j=1,\dotsc,M.
\end{equation}
Being only a partial order on the set of all $(\bfalpha,\bfmu)$-completions, there is no immediate guarantee that a minimal completion with respect to this order exists.
Nevertheless, we show that one does in fact exist, by constructing it explicitly:
\begin{theorem}
\label{theorem.main result 2}
Let $\bfalpha=\{\alpha_m\}_{m=1}^M$ and $\bfmu=\{\mu_{n}\}_{n=1}^N$ be nonnegative and nonincreasing with $M\leq N$.
For any $k=1,\dotsc,M$, given $\set{\beta_m}_{m=k+1}^{M}$ define
\begin{equation*}
\beta_k:=\max\Biggset{t\in\bbR: \sum_{m=j}^{k}(t-\alpha_{m-j+1})^++\sum_{m=k+1}^{M}(\beta_m-\alpha_{m-j+1})^+\leq\sum_{n=j}^{N}\mu_n,\ \forall j=1,\dotsc,k}.
\end{equation*}
Then $\set{\beta_m}_{m=1}^{M}$ is a well-defined $(\bfalpha,\bfmu)$-completion and moreover is the minimal such completion with respect to majorization: if $\set{\lambda_m}_{m=1}^{M}$ is any $(\bfalpha,\bfmu)$-completion then $\set{\beta_m}_{m=1}^{M}\preceq\set{\lambda_m}_{m=1}^{M}$.
\end{theorem}
Here, we have assumed $M\leq N$ since it makes the proof of Theorem~\ref{theorem.main result 2} slightly cleaner; to apply the result in the case where $N<M$, simply define $\mu_n:=0$ for all $n=N+1,\dotsc,M$.

Note that the minimal completion $\set{\beta_m}_{m=1}^{M}$ given by Theorem~\ref{theorem.main result 2} is obviously unique.
Indeed, if both $\set{\beta_m}_{m=1}^{N}$ and $\set{\lambda_m}_{m=1}^{M}$ are minimal completions then $\set{\beta_m}_{m=1}^{N}\preceq\set{\lambda_m}_{m=1}^{M}$ and $\set{\lambda_m}_{m=1}^{N}\preceq\set{\beta_m}_{m=1}^{M}$.
Thus, \smash{$\sum_{m=1}^{j}\beta_m=\sum_{m=1}^{j}\lambda_m$} for all $j=1,\dotsc,M$, implying $\beta_m=\lambda_m$ for all $m$.
To see why this minimal completion is optimally tight, note that letting $j=1$ in~\eqref{equation.majorization from bottom} gives $\beta_1\leq\lambda_1$ for all $(\bfalpha,\bfmu)$-completions $\set{\lambda_m}_{m=1}^{M}$,
meaning that of all possible such completions, the maximum value of $\set{\beta_m}_{m=1}^{M}$ is as small as possible.
At the same time, the minimum value of $\set{\beta_m}_{m=1}^{M}$ is as large as possible: subtracting the inequalities in~\eqref{equation.majorization from bottom} from the equality $\sum_{m=1}^{M}\beta_m=\sum_{m=1}^{M}\lambda_m$ gives the equivalent inequalities:
\begin{equation*}
\sum_{m=j}^{M}\lambda_m\leq\sum_{m=j}^{M}\beta_m,\quad\forall j=1,\dotsc,M.
\end{equation*}
In the special case where $j=M$, we see that any $(\bfalpha,\bfmu)$-completion $\set{\lambda_m}_{m=1}^{M}$ necessarily satisfies $\lambda_M\leq\beta_M$, as claimed.
Together, these facts imply that $\beta_1/\beta_M\leq \lambda_1/\lambda_M$ for any such $\set{\lambda_m}_{m=1}^{M}$, meaning $\set{\beta_m}_{m=1}^{M}$ has the smallest condition number of any $(\bfalpha,\bfmu)$-completion.
Moreover, $\set{\beta_m}_{m=1}^{M}$ is optimal in an even stronger sense.
To be clear, using some of the techniques of this paper, one can show that there sometimes exists other $(\bfalpha,\bfmu)$-completions $\set{\lambda_m}_{m=1}^{M}$ that have the same condition number as $\set{\beta_m}_{m=1}^{M}$, having $\lambda_1=\beta_1$ and $\lambda_M=\beta_M$ but not $\lambda_m=\beta_m$ for all $m=2,\dotsc,M-1$.
Nevertheless, $\set{\beta_m}$ is a better completion than these: being a minimum with respect to majorization~\eqref{equation.majorization from bottom},
the classical theory of \textit{Schur-convexity} tells us that $\sum_{m=1}^{M}f(\beta_m)\leq\sum_{m=1}^{M}f(\lambda_m)$ for any convex function $f$.
In particular, letting $f(x)=x^2$ we see that $\set{\beta_m}_{m=1}^{M}$ has minimal frame potential.
Moreover, if $\beta_M>0$ then letting $f(x)=1/x$ gives that $\set{\beta_m}_{m=1}^{M}$ has minimal mean squared reconstruction error.

Before moving on to the proofs of Theorems~\ref{theorem.main result 1} and~\ref{theorem.main result 2}, we take a moment to put Theorem~\ref{theorem.main result 2} in the context of the literature, specifically the recent work of~\cite{MasseyRS14b}.
To be clear, the problem addressed by Theorem~\ref{theorem.main result 2}---to provide an algorithm for computing the $(\bfalpha,\bfmu)$-completion of a given frame which is optimal with respect to majorization---was first posed in~\cite{MasseyRS13}.
This same paper contained a partial solution to this problem.
An even better partial solution was given in a follow-up paper by these same authors~\cite{MasseyRS14a}.
Shortly thereafter, they wrote a second follow-up paper~\cite{MasseyRS14b} that provides a complete solution to this problem;
it is against this most recent work that we compare our own.

The algorithm given in~\cite{MasseyRS14b} for computing the optimal $(\bfalpha,\bfmu)$-completion is completely different from Theorem~\ref{theorem.main result 2}.
Moreover, it is proven in a completely different way.
This is not surprising: \cite{MasseyRS14b} derives its algorithm directly without having access to the succinct and powerful characterization of $(\bfalpha,\bfmu)$-completions given in Theorem~\ref{theorem.main result 1}.
To be precise, Theorem~3.7 of~\cite{MasseyRS14b} shows their algorithm---given in Proposition~3.6 of that same paper---indeed computes optimal $(\bfalpha,\bfmu)$-completions.
To understand their algorithm in detail, the interested reader must also consider Remark~2.13, Remark~3.2 and Theorem~3.4.
By comparison, the algorithm of Theorem~\ref{theorem.main result 2} is much shorter as a statement, and is self-contained.
This is one advantage of Theorem~\ref{theorem.main result 2}.
A second advantage is the nature of its proof: though the proofs of both Theorem~\ref{theorem.main result 2} and Theorem~3.7 of~\cite{MasseyRS14b} are very technical, the former has a nice geometric motivation.
Indeed, as discussed in Section~3, we construct an optimal completion by \textit{water filling}---a well-known spectral technique from the theory of communications---subject to the constraints of Theorem~\ref{theorem.main result 1}.
In other respects, neither algorithm has a clear advantage.
Both algorithms are computing the same spectrum since, as noted above, the optimal completion is unique.
Moreover, it is hard to determine exactly which algorithm is more computationally efficient: at the end of this paper, we give an explicit example which illustrates exactly how we implement Theorem~\ref{theorem.main result 2}, and then discuss how we implement it in general; no example is given in~\cite{MasseyRS14b}, and we were not able to find or determine a decent operation count for that algorithm.
Nevertheless, both algorithms seem very fast, and can be performed by hand in spaces of sufficiently low dimension.
And, moving forward, we believe that both our proof techniques as well as those of~\cite{MasseyRS14b} will be useful in future research.

\section{Characterizing all completions: Proving Theorem~\ref{theorem.main result 1}}

In this section we characterize the spectra of all possible completions of a positive semidefinite matrix $\bfA\in\bbF^{M\times M}$ with vectors $\set{\bfphi_n}_{n=1}^{N}$ of given lengths $\bfmu=\set{\mu_n}_{n=1}^{N}$.
To be precise, let $\bfalpha=\set{\alpha_m}_{m=1}^{M}$ denote the nonnegative spectrum of $\bfA$ and assume without loss of generality that both $\set{\alpha_m}_{m=1}^{M}$ and $\set{\mu_n}_{n=1}^{N}$ are arranged in nonincreasing order.
We characterize all possible $(\bfalpha,\bfmu)$-completions, that is, the spectra $\set{\lambda_m}_{m=1}^{M}$ of all operators of the form \smash{$\bfA+\sum_{n=1}^{N}\bfphi_n^{}\bfphi_n^*$} where $\set{\bfphi_n}_{n=1}^{N}$ are vectors in $\bbF^M$ that satisfy $\norm{\bfphi_n}^2=\mu_n$ for all $n$.
Here, note that by conjugating by a unitary matrix whose columns are eigenvectors of $\bfA$ we may assume without loss of generality that $\bfA$ is diagonal.
In particular, our characterization of $\set{\lambda_m}_{m=1}^{M}$ will not depend on $\bfA$ per se, but rather, on its spectrum $\bfalpha$.

To obtain some necessary conditions, fix any $\set{\bfphi_n}_{n=1}^{N}$ in $\bbF^M$ with $\norm{\bfphi_n}^2=\mu_n$ for all $n$,
and let $\set{\lambda_m}_{m=1}^{M}$ be the nonnegative nonincreasing spectrum of \smash{$\bfA+\sum_{n=1}^{N}\bfphi_n^{}\bfphi_n^*$}.
The key idea is that for any given $P=0,\dotsc,N$ we also consider the nonnegative nonincreasing spectrum $\set{\lambda_{P;m}}_{m=1}^{M}$ of the \textit{$P$th partial completion} $\bfA+\sum_{n=1}^{P}\bfphi_n^{}\bfphi_n^*$.
Letting $P=0$ and $P=N$ gives $\lambda_{0;m}=\alpha_m$ and $\lambda_{N;m}=\lambda_n$ for all $m$, respectively.
Moreover, the trace of the $P$th partial completion is necessarily
\begin{equation*}
\sum_{m=1}^M\lambda_{P;m}
=\Tr\Biggparen{\bfA+\sum_{n=1}^{P}\bfphi_n^{}\bfphi_n^*}
=\Tr(\bfA)+\sum_{n=1}^{P}\Tr(\bfphi_n^*\bfphi_n^{})
=\sum_{m=1}^{M}\alpha_m+\sum_{n=1}^{P}\norm{\bfphi_n}^2
=\sum_{m=1}^{M}\alpha_m+\sum_{n=1}^{P}\mu_n.
\end{equation*}
Finally, for any $P=1,\dotsc,N$, the $P$th partial completion is obtained by adding the rank-one self-adjoint operator $\bfphi_P^{}\bfphi_P^*$ to the $(P-1)$th partial completion and so a well-known classical result from matrix analysis implies that $\set{\lambda_{P;m}}_{m=1}^{M}$ necessarily \textit{interlaces over} \smash{$\set{\lambda_{P-1;m}}_{m=1}^{M}$} in the sense that $\lambda_{P;m+1}\leq\lambda_{P-1;m}\leq\lambda_{P;m}$ for all $m=1,\dotsc,M$, under the convention that $\lambda_{P;M+1}:=0$.
To elaborate on this last condition, note that for any $P=0,\dotsc,N$ we have $\bfA+\sum_{n=1}^{P}\bfphi_n^{}\bfphi_n^*=\bfX_P^{}\bfX_P^*$ where $X_P$ is the $M\times(M+P)$ matrix obtained by concatenating the $M\times M$ matrix \smash{$\bfA^{\frac12}$} with the $P$ column vectors $\set{\bfphi_n}_{n=1}^{P}$.
Since $M+P\geq M$, the spectrum of the corresponding Gram matrix $\bfX_P^*\bfX_P^{}$ is a zero-padded version of the spectrum of $\bfX_P^{}\bfX_P^*$.
That is, $\bfX_P^*\bfX_P^{}$ has spectrum $\set{\lambda_{P;m}}_{m=1}^{M+P}$ provided we define $\lambda_{P;m}:=0$ when $m>M$.
Moreover, for any $P=1,\dotsc,N$ the Gram matrix $\bfX_{P-1}^*\bfX_{P-1}^{}$ is the first principal $(P-1)\times(P-1)$ submatrix of $\bfX_P^*\bfX_P^{}$.
At this point, the famous Cauchy interlacing theorem implies the eigenvalues of $\bfX_{P-1}^*\bfX_{P-1}^{}$ interlace in those of $\bfX_P^*\bfX_P^{}$,
namely that $\lambda_{P;m+1}\leq\lambda_{P-1;m}\leq\lambda_{P;m}$ for all $m=1,\dotsc,M+P-1$.
This is precisely the interlacing condition we gave above, provided we realize it is superfluous for all $m>M$, requiring $0\leq 0\leq0$.

Together, any sequence of spectra $\set{\lambda_{P;m}}_{m=1}^{M}$ that satisfies these conditions is known as a sequence of \textit{eigensteps}:
\begin{definition}
\label{definition.eigensteps}
For any nonnegative nonincreasing sequences $\bfalpha=\set{\alpha_m}_{m=1}^{M}$, $\bflambda=\set{\lambda_m}_{m=1}^{M}$ and $\bfmu=\set{\mu_n}_{n=1}^{N}$,
a sequence of nonincreasing sequences $\set{\set{\lambda_{P;m}}_{m=1}^{M}}_{P=0}^{N}$ is a \textit{sequence of eigensteps from $\bfalpha$ to $\bflambda$ with lengths $\bfmu$} if
\begin{enumerate}
\renewcommand{\labelenumi}{(\roman{enumi})}
\item
$\lambda_{0;m}=\alpha_m$ for all $m=1,\dotsc,M$,
\item
$\lambda_{N;m}=\lambda_m$ for all $m=1,\dotsc,M$,
\item
$\sum_{m=1}^M\lambda_{P;m}=\sum_{m=1}^{M}\alpha_m+\sum_{n=1}^{P}\mu_n$ for all $P=0,\dotsc,N$,
\item
$\lambda_{P;m+1}\leq\lambda_{P-1;m}\leq\lambda_{P;m}$ for all $m=1,\dotsc,M$, $P=1,\dotsc,N$; here $\lambda_{P;M+1}:=0$.
\end{enumerate}
\end{definition}
In the special case where $\alpha_m=0$ for all $m$, the above definition reduces to the definition of eigensteps that was introduced in~\cite{CahillFMPS13}.
Having that any $(\bfalpha,\bfmu)$-completion $\bflambda$ yields eigensteps, we can quickly prove the ``only if" direction of Theorem~\ref{theorem.main result 1},
namely that $\lambda_m\geq\alpha_m$ for all $m$ and that~\eqref{equation.completion Schur-Horn} holds:

\begin{proof}[Proof of the ($\Rightarrow$) direction of Theorem~\ref{theorem.main result 1}]
Let $\set{\lambda_m}_{m=1}^{M}$ be any $(\bfalpha,\bfmu)$-completion, meaning there exists a positive semidefinite matrix $\bfA\in\bbF^{M\times M}$ whose spectrum is $\set{\alpha_m}_{m=1}^{M}$ as well as a sequence of vectors $\set{\bfphi_n}_{n=1}^{N}$ in $\bbF^M$ with $\norm{\bfphi_n}^2=\mu_n$ for all $n=1,\dotsc,N$ such that $\set{\lambda_m}_{m=1}^{M}$ is the spectrum of \smash{$\bfA+\sum_{n=1}^{N}\bfphi_n^{}\bfphi_n^*$}.
As noted above, for any $P=0,\dotsc,N$ letting $\set{\lambda_{P;m}}_{m=1}^{M}$ denote the nonnegative nonincreasing spectrum of $\bfA+\sum_{n=1}^{P}\bfphi_n^{}\bfphi_n^*$ yields a sequence of eigensteps, cf.\ Definition~\ref{definition.eigensteps}.
In particular, eigenstep conditions (i) and (ii) as well as (possibly repeated) use of (iv) gives $\alpha_m=\lambda_{0;m}\leq\lambda_{N;m}=\lambda_m$ for all $m=1,\dotsc,M$, as claimed.
Next, the equality condition of~\eqref{equation.completion Schur-Horn} follows immediately from letting $P=N$ in (iii):
\begin{equation*}
\sum_{m=1}^{M}\lambda_m
=\sum_{m=1}^M\lambda_{N;m}
=\sum_{m=1}^{M}\alpha_m+\sum_{n=1}^{N}\mu_n.
\end{equation*}
To prove the inequality conditions in~\eqref{equation.completion Schur-Horn}, note that for any $j=1,\dotsc,N$, subtracting the $P=j-1$ instance of (iii) from the $P=N$ instance of (iii) gives
\begin{equation*}
\sum_{m=1}^M(\lambda_m-\lambda_{j-1;m})
=\sum_{m=1}^M\lambda_{N;m}-\sum_{m=1}^M\lambda_{j-1;m}
=\Biggparen{\sum_{m=1}^{M}\alpha_m+\sum_{n=1}^{N}\mu_n}-\Biggparen{\sum_{m=1}^{M}\alpha_m+\sum_{n=1}^{j-1}\mu_n}
=\sum_{n=j}^{N}\mu_n,
\quad 1\leq j\leq N.
\end{equation*}
Continuing, note that the upper bounds in (iv) give $\lambda_{j-1;m}\leq\lambda_{N;m}=\lambda_m$ for all $j=1,\dotsc,N$ and $m=1,\dotsc,M$, and so
\begin{equation*}
\sum_{m=j}^M(\lambda_m-\lambda_{j-1;m})
\leq\sum_{m=1}^M(\lambda_m-\lambda_{j-1;m})
=\sum_{n=j}^{N}\mu_n,
\quad 1\leq j\leq\min\set{M,N}.
\end{equation*}
Meanwhile, the lower bound in (iv) gives $\lambda_{j-1;m}\leq\lambda_{j-2;m-1}\leq\dotsb\leq\lambda_{0;m-(j-1)}=\alpha_{m-j+1}$ for all $m=j,\dotsc,M$.
To summarize, for any $m=j,\dotsc,M$ we have both $0\leq\lambda_m-\lambda_{j-1;m}$ and $\lambda_m-\alpha_{m-j+1}\leq\lambda_m-\lambda_{j-1;m}$, implying
\begin{equation}
\label{equation.proof of necessity 1}
\sum_{m=j}^{M}(\lambda_m-\alpha_{m-j+1})^+
=\sum_{m=j}^{M}\max\set{0,\lambda_m-\alpha_{m-j+1}}
\leq\sum_{m=j}^M(\lambda_m-\lambda_{j-1;m})
\leq\sum_{n=j}^{N}\mu_n,
\quad 1\leq j\leq\min\set{M,N}.
\end{equation}
In the case where $M\leq N$,~\eqref{equation.proof of necessity 1} yields all the claimed inequality conditions of~\eqref{equation.completion Schur-Horn}.
In the case where $N<M$,~\eqref{equation.proof of necessity 1} still implies the inequalities in~\eqref{equation.completion Schur-Horn} hold for all $j=1,\dotsc,N$.
What remains is the case where $N<j\leq M$; for such $j$, the right-hand side of the inequality in~\eqref{equation.completion Schur-Horn} is defined to be zero, being an empty sum.
As such, the corresponding inequality can only hold provided $(\lambda_m-\alpha_{m-j+1})^+=0$ for all $m=j,\dotsc,M$.
This follows from repeatedly applying the lower bound in (iv): since $N\leq j-1\leq m-1$ we have
$\lambda_m=\lambda_{N;m}\leq\lambda_{N-1;m-1}\leq\dotsb\leq\lambda_{0;m-N}=\alpha_{m-N}\leq\alpha_{m-j+1}$.
\end{proof}

Our proof of the ``if" direction of Theorem~\ref{theorem.main result 1} is substantially more involved, and requires several technical lemmas.
The first lemma is a strengthening of one of the main results of~\cite{CahillFMPS13}:
\begin{lemma}
\label{lemma.eigensteps}
For any nonnegative nonincreasing sequences $\bfalpha=\set{\alpha_m}_{m=1}^{M}$, $\bflambda=\set{\lambda_m}_{m=1}^{M}$ and $\bfmu=\set{\mu_n}_{n=1}^{N}$,
$\bflambda$ is an $(\bfalpha,\bfmu)$-completion (Definition~\ref{definition.completion}) if and only if there exists a sequence of eigensteps from $\bfalpha$ to $\bflambda$ with lengths $\bfmu$ (Definition~\ref{definition.eigensteps}).
\end{lemma}

\begin{proof}
The reasons why eigensteps necessarily exist for any $(\bfalpha,\bfmu)$-completion were discussed above: $\set{\lambda_{P;m}}_{m=1}^{M}$ is defined to be the nonincreasing spectrum of \smash{$\bfA+\sum_{n=1}^{P}\bfphi_n^{}\bfphi_n^*$}.
Conversely, suppose there exists a sequence of eigensteps $\set{\set{\lambda_{P;m}}_{m=1}^{M}}_{P=0}^{N}$ from $\bfalpha$ to $\bflambda$ with lengths $\bfmu$.
To construct $\bfA$ and $\set{\bfphi_n}_{n=1}^{N}$ we exploit Theorem~2 of~\cite{CahillFMPS13} which constructs frame vectors from eigensteps whose initial spectrum is identically zero.

In particular, taking any fixed $\beta\geq\max\set{0,\mu_1-\alpha_M}$, we claim defining $\set{\set{\kappa_{P;m}}_{m=1}^{M}}_{P=0}^{M+N}$ and $\set{\nu_n}_{n=1}^{M+N}$ by
\begin{equation}
\label{equation.proof of eigensteps 1}
\kappa_{P;m}:=\left\{\begin{array}{ll}0, &P<m,\\\alpha_m+\beta,& m\leq P\leq M,\\\lambda_{P-M;m}+\beta,& M<P,\end{array}\right.
\qquad
\nu_n=\left\{\begin{array}{ll}\alpha_n+\beta,& n\leq M,\\\mu_{n-M},& M< n,\end{array}\right.
\end{equation}
yields a sequence of eigensteps from $\set{0}_{m=1}^{M}$ to $\set{\lambda_m+\beta}_{m=1}^{M}$ with lengths $\set{\nu_n}_{n=1}^{M+N}$;
here the choice of $\beta$ ensures that $\set{\nu_n}_{n=1}^{M+N}$ is nonnegative and nonincreasing.
Indeed, $\kappa_{0;m}=0$ and $\kappa_{M+N;m}=\lambda_{N;m}+\beta=\lambda_m+\beta$ for all $m=1,\dotsc,M$ and so these sequences satisfy conditions (i) and (ii) of Definition~\ref{definition.eigensteps}.
Next, in this setting condition (iii) becomes $\sum_{m=1}^{M}\kappa_{P;m}=\sum_{n=1}^{P}\nu_n$ for all $P=0,\dotsc,M+N$.
For $P\leq M$ this holds since $\sum_{m=1}^{M}\kappa_{P;m}=\sum_{m=1}^{P}(\alpha_m+\beta)=\sum_{n=1}^{P}\nu_n$.
For $P>M$, recall our assumption that $\set{\set{\lambda_{P;m}}_{m=1}^{M}}_{P=0}^{N}$ is a sequence of eigensteps from $\bfalpha$ to $\bflambda$ with lengths $\bfmu$; condition (iii) of this assumption gives
\begin{equation*}
\sum_{m=1}^{M}\kappa_{P;m}
=\sum_{m=1}^{M}(\lambda_{P-M;m}+\beta)
=\Biggparen{\sum_{m=1}^{M}\alpha_m+\sum_{n=1}^{P-M}\mu_n}+M\beta
=\sum_{m=1}^{M}(\alpha_m+\beta)+\sum_{n=M+1}^{P}\mu_{n-M}
=\sum_{n=1}^{P}\nu_n.
\end{equation*}
Finally, we prove (iv), namely that $\kappa_{P;m+1}\leq\kappa_{P-1;m}\leq\kappa_{P;m}$ for all $P=1,\dotsc,M+N$ and $m=1,\dotsc,M$.
For $P\leq M$, this inequality holds for different reasons depending on the relationship between $m$ and $P$:
for $m\leq P-1$ it becomes $\alpha_{m+1}+\beta\leq\alpha_m+\beta\leq\alpha_m+\beta$, which follows from the fact that $\set{\alpha_m}_{m=1}^{M}$ is nonnegative and nonincreasing;
for $m=P$ it becomes $0\leq 0\leq\alpha_m+\beta$ which holds since $\beta\geq 0$;
for $m>P$ it becomes $0\leq 0\leq 0$.
Meanwhile, (iv) also holds in the case where $P>M$ since we are simply adding $\beta$ to our assumed version of (iv):
$\kappa_{P-1;m}=\lambda_{P-1-M;m}+\beta\leq\lambda_{P-M;m}+\beta=\kappa_{P;m}$ for all $m=1,\dotsc,M$
and $\kappa_{P;m+1}=\lambda_{P-M;m+1}+\beta\leq\lambda_{P-1-M;m}+\beta=\kappa_{P-1;m}$ for all $m=1,\dotsc,M-1$.

Having that~\eqref{equation.proof of eigensteps 1} defines a sequence of eigensteps from $\set{0}_{m=1}^{M}$ to $\set{\lambda_m+\beta}_{m=1}^{M}$ with lengths $\set{\nu_n}_{n=1}^{M+N}$, Theorem~2 of~\cite{CahillFMPS13} gives the existence of a sequence of vectors $\set{\bfpsi_n}_{n=1}^{M+N}$ with $\norm{\bfpsi_n}^2=\nu_n$ for all $n$ which also has the property that $\set{\kappa_{P;m}}_{m=1}^{M}$ is the spectrum of the $P$th partial frame operator \smash{$\bfPsi_P^{}\bfPsi_P^*=\sum_{n=1}^{P}\bfpsi_n^{}\bfpsi_n^*$} for any given $P=1,\dotsc,M+N$.
Let $\bfA=\bfPsi_M^{}\bfPsi_M^*-\beta\bfI$ which has spectrum $\set{\kappa_{M;m}-\beta}_{m=1}^{M}=\set{\alpha_m}_{m=1}^{M}$.
Let $\bfphi_n:=\bfpsi_{M+n}$ for all $n=1,\dotsc,N$, meaning $\norm{\bfphi_n}^2=\norm{\bfpsi_{M+n}}^2=\nu_{M+n}=\mu_n$ for all such $n$.
Moreover, the operator
\begin{equation*}
\bfA+\sum_{n=1}^{N}\bfphi_n^{}\bfphi_n^*
=(\bfPsi_M^{}\bfPsi_M^*-\beta\bfI)+\sum_{n=1}^{N}\bfpsi_{M+n}^{}\bfpsi_{M+n}^*
=\sum_{n=1}^M\bfpsi_n^{}\bfpsi_n^*-\beta\bfI+\sum_{n=M+1}^{M+N}\bfpsi_n^{}\bfpsi_n^*
=\bfPsi_{M+N}^{}\bfPsi_{M+N}^*-\beta\bfI
\end{equation*}
has spectrum $\set{\kappa_{M+N;m}-\beta}_{m=1}^{M}=\set{\lambda_m}_{m=1}^{M}$, meaning $\set{\lambda_m}_{m=1}^{M}$ is an $(\bfalpha,\bfmu)$-completion.
\end{proof}

To summarize, if we want to show a given spectrum $\set{\lambda_m}_{m=1}^{M}$ is an $(\bfalpha,\bfmu)$-completion it suffices to construct a corresponding sequence of eigensteps.
In the remainder of this section, we discuss how condition~\eqref{equation.completion Schur-Horn} of Theorem~\ref{theorem.main result 1} lends itself to an iterative construction of such eigensteps.
Here, the main idea is a nontrivial generalization of the \textit{Top Kill} algorithm of~\cite{FickusMPS13}.

Following~\cite{FickusMPS13}, we visualize a nonnegative nonincreasing spectra $\set{\lambda_m}_{m=1}^{M}$ as a pyramid:
each eigenvalue $\lambda_m$ is represented as a horizontal stone block of length $\lambda_m$ and height $1$ that provides a foundation for the block of length $\lambda_{m+1}$ that lies on top of it.
In order to take one eigenstep backwards, we want a nonnegative nonincreasing spectrum $\set{\kappa_m}_{m=1}^{M}$ such that $\sum_{m=1}^{M}\kappa_m=\sum_{m=1}^{M}\lambda_m-\mu_N$ and such that $\set{\lambda_m}_{m=1}^{M}$ interlaces over $\set{\kappa_m}_{m=1}^{M}$.
In terms of pyramids, the trace condition means we form $\set{\kappa_m}_{m=1}^{M}$ by chipping away $\mu_N$ units of stone from $\set{\lambda_m}_{m=1}^{M}$.
Moreover, the interlacing condition means we can only remove the portion of a $\lambda_m$ block that is not covered by the corresponding $\lambda_{m+1}$ block.

Moving beyond the intuition of~\cite{FickusMPS13} so as to address the completion problem, we now further envision that these pyramids encase a pyramidal foundation corresponding to the initial spectrum $\set{\alpha_m}_{m=1}^{M}$.
Our goal is to reveal this foundation via an $N$-stage excavation of $\set{\lambda_m}_{m=1}^{M}$; each stage converts eigensteps $\set{\lambda_{P,m}}_{m=1}^{M}$ into $\set{\lambda_{P-1,m}}_{m=1}^{M}$ for some $P=1,\dotsc,N$.
It turns out that accomplishing this goal requires careful planning.
Indeed, one might be tempted to first completely excavate the highest level of the foundation, then proceed onto the second-highest level, etc.; it turns out that this approach sometimes fails to reveal the entire foundation in $N$ stages, even when the conditions of Theorem~\ref{theorem.main result 1} are satisfied~\cite{Poteet12}.
A better method---one we can prove always works---is to always prioritize the removal of stone that buries the foundation most deeply.
In particular, in the next lemma, for any $m=1,\dotsc,M$ and $p=1,\dotsc,M+1$, we consider the $p$th ``chopped spectrum" obtained by removing the portion of $\lambda_m$ that is not covered by $\lambda_{m+1}$ and which lies at least $p$ layers above its foundation $\set{\alpha_m}_{m=1}^{M}$.
To take one eigenstep backwards from $\set{\lambda_m}_{m=1}^{M}$, we then choose a spectrum $\set{\kappa_m}_{m=1}^{M}$ that lies between two consecutive ``chops" and has the requisite trace.

\begin{lemma}
\label{lemma.chopped spectra}
Let $M$ and $N$ be positive integers and let $\set{\alpha_m}_{m=1}^{M}$, $\set{\lambda_m}_{m=1}^{M}$ and $\set{\mu_n}_{n=1}^{N}$ be any nonnegative  nonincreasing sequences with $\alpha_m\leq\lambda_m$ for all $m$ that also satisfy~\eqref{equation.completion Schur-Horn}.
For any $p=1,\dotsc,M+1$, define the $p$th chopped spectrum $\set{\eta_{p;m}}_{m=1}^{M}$ of $\set{\lambda_m}_{m=1}^{M}$ with respect to $\set{\alpha_m}_{m=1}^{M}$ as
\begin{equation}
\label{equation.proof of Chop Kill step 1}
\eta_{p;m}:=\max\set{\lambda_{m+1},\min\set{\lambda_m,\alpha_{m-p+1}}},\quad\forall m=1,\dotsc,M,
\end{equation}
under the conventions that $\lambda_{M+1}:=0$ and $\alpha_m:=\infty$ for all $m\leq0$.
For any $m=1,\dotsc,M$, the sequence \smash{$\set{\eta_{p:m}}_{p=1}^{M+1}$} is nondecreasing with $\eta_{1;m}=\max\set{\lambda_{m+1},\alpha_m}$ and $\eta_{M+1;m}=\lambda_m$.
Moreover, there exists an index $p=1,\dotsc,M$ and a sequence $\set{\kappa_m}_{m=1}^{M}$ such that
\begin{equation}
\label{equation.proof of Chop Kill step 7}
\sum_{m=1}^{M}\kappa_m=\sum_{m=1}^{M}\alpha_m+\sum_{n=1}^{N-1}\mu_n,\qquad \eta_{p:m}\leq\kappa_m\leq\eta_{p+1;m},\quad\forall m=1,\dotsc,M.
\end{equation}
\end{lemma}

\begin{proof}
For any $m=1,\dotsc,M$, the fact that \smash{$\set{\eta_{p:m}}_{p=1}^{M+1}$} is nondecreasing follows from the fact that \smash{$\set{\alpha_m}_{m=-\infty}^{M}$} is nonincreasing: \smash{$\alpha_{m-p+1}\leq\alpha_{m-p}$} and thus \smash{$\eta_{p;m}=\max\set{\lambda_{m+1},\min\set{\lambda_m,\alpha_{m-p+1}}}\leq\max\set{\lambda_{m+1},\min\set{\lambda_m,\alpha_{m-p}}}=\eta_{p+1;m}$} for all $p=1,\dotsc,M$.
Next, since $\alpha_m\leq\lambda_m$ for all $m$, the $p=1$ case of~\eqref{equation.proof of Chop Kill step 1} reduces to
\begin{equation*}
\eta_{1;m}
=\max\set{\lambda_{m+1},\min\set{\lambda_m,\alpha_{m}}}
=\max\set{\lambda_{m+1},\alpha_m},\quad\forall m=1,\dotsc,M,
\end{equation*}
as claimed.
Similarly, since $\set{\lambda_m}_{m=1}^{M+1}$ is nonincreasing and $\alpha_m:=\infty$ for all $m\leq0$, the $p=M+1$ case of~\eqref{equation.proof of Chop Kill step 1} becomes
\begin{equation*}
\eta_{M+1;m}=\max\set{\lambda_{m+1},\min\set{\lambda_m,\alpha_{m-M}}}
=\max\set{\lambda_{m+1},\min\set{\lambda_m,\infty}}
=\max\set{\lambda_{m+1},\lambda_m}
=\lambda_m,
\quad\forall m=1,\dotsc,M.
\end{equation*}

To prove there exists $p$ and $\set{\kappa_m}_{m=1}^{M}$ such that~\eqref{equation.proof of Chop Kill step 7} holds,
consider the trace $\tau_p:=\sum_{m=1}^{M}\eta_{p;m}$ of each chopped spectrum.
Since $\set{\eta_{p:m}}_{p=1}^{M+1}$ is nondecreasing for each $m=1,\dotsc,M$ we know that $\set{\tau_p}_{p=1}^{M+1}$ is also nondecreasing.
Moreover, the equality condition in our assumption~\eqref{equation.completion Schur-Horn} along with the fact that $\mu_N\geq0$ imply that $\tau_{M+1}$ is an upper bound for the quantity $\sigma:=\sum_{m=1}^{M}\alpha_m+\sum_{n=1}^{N-1}\mu_n$, which is intended to be the trace of our desired spectrum $\set{\kappa_m}_{m=1}^{M}$:
\begin{equation*}
\tau_{M+1}
=\sum_{m=1}^{M}\eta_{M+1;m}
=\sum_{m=1}^{M}\lambda_m
=\sum_{m=1}^{M}\alpha_m+\sum_{n=1}^{N}\mu_n
\geq\sum_{m=1}^{M}\alpha_m+\sum_{n=1}^{N-1}\mu_n
=\sigma.
\end{equation*}
We further claim $\sigma$ is bounded below by $\tau_1$.
To see this, first note that
\begin{equation*}
\tau_1
=\sum_{m=1}^{M}\eta_{1;m}
=\sum_{m=1}^{M}\max\set{\lambda_{m+1},\alpha_m}
=\sum_{m=1}^{M}\max\set{\lambda_{m+1}-\alpha_m,0}+\sum_{m=1}^{M}\alpha_m
=\sum_{m=1}^{M}(\lambda_{m+1}-\alpha_m)^++\sum_{m=1}^{M}\alpha_m.
\end{equation*}
Next, recall that $\lambda_{M+1}:=0$ and so $(\lambda_{M+1}-\alpha_M)^+=0$, implying
\begin{equation*}
\tau_1
=\sum_{m=1}^{M-1}(\lambda_{m+1}-\alpha_m)^++\sum_{m=1}^{M}\alpha_m
=\sum_{m=2}^{M}(\lambda_m-\alpha_{m-1})^++\sum_{m=1}^{M}\alpha_m.
\end{equation*}
Invoking our assumption~\eqref{equation.completion Schur-Horn} in the $j=2$ case and then using the fact that $\mu_1\geq\mu_N$ then gives our claim:
\begin{equation*}
\tau_1
=\sum_{m=2}^{M}(\lambda_m-\alpha_{m-1})^++\sum_{m=1}^{M}\alpha_m
\leq\sum_{n=2}^{N}\mu_n+\sum_{m=1}^{M}\alpha_m
\leq\sum_{n=1}^{N-1}\mu_n+\sum_{m=1}^{M}\alpha_m
=\sigma.
\end{equation*}
A technicality: using $j=2$ in~\eqref{equation.completion Schur-Horn} implicitly assumes that $M\geq2$; fortunately, the above inequality also holds when $M=1$ since in that case $\sum_{m=2}^{M}(\lambda_m-\alpha_{m-1})^+=0\leq\sum_{n=2}^{N}\mu_n$.

Having that \smash{$\set{\tau_p}_{p=1}^{M+1}$} is nondecreasing with $\tau_1\leq\sigma\leq \tau_{M+1}$, there exists at least one index $p$ with $1\leq p\leq M$ and such that $\tau_p\leq\sigma\leq\tau_{p+1}$.
Fixing such an index $p$, let $\set{\kappa_m}_{m=1}^{M}$ be any sequence such that~\eqref{equation.proof of Chop Kill step 7} holds.
Such a sequence always exists: since $\tau_p\leq\sigma\leq\tau_{p+1}$, there exists $t\in[0,1]$ such that $\sigma=\tau_p+(\tau_{p+1}-\tau_p)t$ and we can let $\kappa_m:=\eta_{p;m}+(\eta_{p+1;m}-\eta_{p;m})t$, for example.
\end{proof}

To recap, our goal for the rest of this section is to prove the $(\Leftarrow)$ direction of Theorem~\ref{theorem.main result 1}.
Here, $\set{\lambda_m}_{m=1}^{M}$, $\set{\alpha_m}_{m=1}^{M}$ and $\set{\mu_n}_{n=1}^{N}$ are nonnegative nonincreasing sequences that satisfy~\eqref{equation.completion Schur-Horn} with $\lambda_m\geq\alpha_m$ for all $m$.
In light of Lemma~\ref{lemma.eigensteps}, it suffices to construct a corresponding sequence of eigensteps from $\set{\alpha_m}_{m=1}^{M}$ to $\set{\lambda_m}_{m=1}^{M}$.
Inspired by the Top Kill algorithm of~\cite{FickusMPS13}, we construct these eigensteps iteratively, working backwards from $\set{\lambda_m}_{m=1}^{M}$ to $\set{\alpha_m}_{m=1}^{M}$.
Here, what we really need is a good strategy for ``excavating" a spectrum $\set{\kappa_m}_{m=1}^{M}$ from $\set{\lambda_m}_{m=1}^{M}$.
In Lemma~\ref{lemma.chopped spectra} we propose one such strategy, choosing $\set{\kappa_m}_{m=1}^{M}$ to lie between two chopped spectra of $\set{\lambda_m}_{m=1}^{M}$.
In the next result, we show that any $\set{\kappa_m}_{m=1}^{M}$ chosen in this way is indeed one backwards-eigenstep from $\set{\lambda_m}_{m=1}^{M}$,
having the requisite trace and interlacing properties.
Most importantly, we show that choosing $\set{\kappa_m}_{m=1}^{M}$ in this way ensures that it, like $\set{\lambda_m}_{m=1}^{M}$, satisfies the generalized majorization condition~\eqref{equation.completion Schur-Horn}, albeit for $\set{\mu_n}_{n=1}^{N-1}$ instead of $\set{\mu_n}_{n=1}^{N}$.
As detailed at the end of this section, this allows us to repeatedly use the method of Lemma~\ref{lemma.chopped spectra}, that is, to repeatedly take backwards eigensteps, to arrive at $\set{\alpha_m}_{m=1}^{M}$.

\begin{lemma}
\label{lemma.Chop Kill step}
Let $M$ and $N$ be positive integers and let $\set{\alpha_m}_{m=1}^{M}$, $\set{\lambda_m}_{m=1}^{M}$ and $\set{\mu_n}_{n=1}^{N}$ be any nonnegative  nonincreasing sequences with $\alpha_m\leq\lambda_m$ for all $m$ that also satisfy~\eqref{equation.completion Schur-Horn}.
Then, for any index $p=1,\dotsc,M$ and sequence $\set{\kappa_m}_{m=1}^{M}$ that satisfy~\eqref{equation.proof of Chop Kill step 7},
we have $\set{\kappa_m}_{m=1}^{M}$ is nonincreasing with $\kappa_m\geq\alpha_m$ for all $m$.
Moreover, $\set{\lambda_m}_{m=1}^{M}$ interlaces over $\set{\kappa_m}_{m=1}^{M}$ and
\begin{equation}
\label{equation.Chop Kill step 2}
\sum_{m=1}^{M}(\kappa_m-\alpha_m)=\sum_{n=1}^{N-1}\mu_n,
\qquad
\sum_{m=j}^{M}(\kappa_m-\alpha_{m-j+1})^+\leq\sum_{n=j}^{N-1}\mu_n,\quad\forall j=1,\dotsc,M.
\end{equation}
Moreover, when $N=1$ we necessarily have $\kappa_m=\alpha_m$ for all $m=1,\dotsc,M$.
\end{lemma}

\begin{proof}
Fix any index $p=1,\dotsc,M$ and sequence $\set{\kappa_m}_{m=1}^{M}$ that satisfy~\eqref{equation.proof of Chop Kill step 7};
by Lemma~\ref{lemma.chopped spectra}, we know at least one such index and spectrum exist.
Note Lemma~\ref{lemma.chopped spectra} also gives
$\max\set{\lambda_{m+1},\alpha_m}
=\eta_{1;m}
\leq\eta_{p;m}
\leq\kappa_m
\leq\eta_{p+1;m}
\leq\eta_{M+1;m}
=\lambda_m$
for all $m=1,\dotsc,M$.
In particular, $\set{\kappa_m}_{m=1}^{M}$ satisfies $\kappa_m\geq\alpha_m$ for all $m$.
This same inequality implies $\set{\kappa_m}_{m=1}^{M}$ satisfies the interlacing condition $\lambda_{m+1}\leq\kappa_m\leq\lambda_m$ for all $m=1,\dotsc,M$, which in turn implies that $\set{\kappa_m}_{m=1}^{M}$ is nonincreasing.
Moreover, the equality in~\eqref{equation.Chop Kill step 2} is simply a rewriting of the equality in our assumption~\eqref{equation.proof of Chop Kill step 7}.
Note that when $N=1$, this equality becomes $\sum_{m=1}^{M}(\kappa_m-\alpha_m)=0$;
when combined with the fact that $\kappa_m\geq\alpha_m$, this implies that in this special case we necessarily have $\kappa_m=\alpha_m$ for all $m$.

The remainder of this proof is devoted to showing that $\set{\kappa_m}_{m=1}^{M}$ satisfies the inequality conditions in~\eqref{equation.Chop Kill step 2}.
The argument is complicated; for a geometric motivation of it, we refer the interested reader to the alternative, longer presentation given in~\cite{Poteet12}.
The key idea is to recognize that for any $\gamma\geq0$ and any $j,m=1,\dotsc,M$ with $j\leq m$, the quantity $(\gamma-\alpha_{m-j+1})^+$ corresponds to the length of the intersection of the intervals $[0,\gamma)$ and $[\alpha_{m-j+1},\infty)$ and moreover, that this intersection can be decomposed according to the partition \smash{$[\alpha_{m-j+1},\infty)=\sqcup_{i=0}^{m-j}[\alpha_{i+1},\alpha_i)$}; here, we continue the convention of defining $\alpha_{M+1}:=0$ and $\alpha_0:=\infty$.
In particular, for any nonnegative sequence $\set{\gamma_m}_{m=1}^{M}$ and any $j=1,\dotsc,M$,
\begin{equation*}
\sum_{m=j}^{M}(\gamma_m-\alpha_{m-j+1})^+
=\sum_{m=j}^{M}\ell\bigset{[0,\gamma_m)\cap[\alpha_{m-j+1},\infty)}
=\sum_{m=j}^{M}\sum_{i=0}^{m-j}\ell\bigset{[0,\gamma_m)\cap[\alpha_{i+1},\alpha_i)}.
\end{equation*}
Making the change of variables $k=m-i$ and then interchanging sums gives
\begin{equation}
\label{equation.proof of Chop Kill step 8}
\sum_{m=j}^{M}(\gamma_m-\alpha_{m-j+1})^+
=\sum_{m=j}^{M}\sum_{k=j}^{m}\ell\bigset{[0,\gamma_m)\cap[\alpha_{m-k+1},\alpha_{m-k})}
=\sum_{k=j}^{M}\sum_{m=k}^{M}\ell\bigset{[0,\gamma_m)\cap[\alpha_{m-k+1},\alpha_{m-k})}.
\end{equation}
We now compare the value of $\sum_{m=k}^{M}\ell\bigset{[0,\gamma_m)\cap[\alpha_{m-k+1},\alpha_{m-k})}$ when $\gamma_m=\kappa_m$ to the value of this same sum when $\gamma_m=\lambda_m$.
This comparison will depend on the relationship between $k$ and $p$, where recall $p$ was chosen so that $\sigma$ satisfies~\eqref{equation.proof of Chop Kill step 7}.
For example, we now show these two sums are equal in the case where $k\leq p-1$.

To be precise, take any $k$ such that $1\leq k\leq p-1$; note this part of the argument is vacuous in the $p=1$ case.
The construction of $\set{\kappa_m}_{m=1}^{M}$ in~\eqref{equation.proof of Chop Kill step 7} along with the definition of the chopped spectra~\eqref{equation.proof of Chop Kill step 1} gives
\begin{equation}
\label{equation.proof of Chop Kill step 9}
\kappa_m
\geq\eta_{p;m}
=\max\set{\lambda_{m+1},\min\set{\lambda_m,\alpha_{m-p+1}}}
\geq\min\set{\lambda_m,\alpha_{m-p+1}}.
\end{equation}
Moreover, since $\set{\lambda_m}_{m=1}^{M}$ interlaces over $\set{\kappa_m}_{m=1}^{M}$, we also have $\kappa_m\leq\lambda_m$.
Note that if $\kappa_m<\lambda_m\leq\alpha_{m-p+1}$, the previous two facts together imply $\lambda_m=\min\set{\lambda_m,\alpha_{m-p+1}}\leq\kappa_m<\lambda_m$, a contradiction.
In particular, if $\kappa_m<\lambda_m$ we necessarily have $\lambda_m>\alpha_{m-p+1}$ at which point~\eqref{equation.proof of Chop Kill step 9} gives $\kappa_m\geq\min\set{\lambda_m,\alpha_{m-p+1}}=\alpha_{m-p+1}$.
To summarize, for any $m=1,\dotsc,M$ we either have that $\kappa_m=\lambda_m$ or that $\alpha_{m-p+1}\leq\kappa_m<\lambda_m$.
Further note that for any $m=k,\dotsc,M$ the fact that $k\leq p-1$ implies $m-k\geq m-p+1$ and so $\alpha_{m-k}\leq\alpha_{m-p+1}$.
Thus, for any such $m$ we either have that the intervals $[0,\kappa_m)$ and $[0,\lambda_m)$ are equal or that both contain the interval $[\alpha_{m-k+1},\alpha_{m-k})$.
This implies
\begin{equation}
\label{equation.proof of Chop Kill step 10}
\sum_{m=k}^{M}\ell\bigset{[0,\kappa_m)\cap[\alpha_{m-k+1},\alpha_{m-k})}
=\sum_{m=k}^{M}\ell\bigset{[0,\lambda_m)\cap[\alpha_{m-k+1},\alpha_{m-k})},\quad 1\leq k\leq p-1.
\end{equation}

Next consider any $k$ with $p+1\leq k\leq M$; this is vacuous when $p=M$.
Here~\eqref{equation.proof of Chop Kill step 1} and~\eqref{equation.proof of Chop Kill step 7} give
\begin{equation}
\label{equation.proof of Chop Kill step 11}
\kappa_m
\leq\eta_{p+1;m}
=\max\set{\lambda_{m+1},\min\set{\lambda_m,\alpha_{m-(p+1)+1}}}
=\max\set{\lambda_{m+1},\min\set{\lambda_m,\alpha_{m-p}}}.
\end{equation}
Since $\set{\lambda_m}_{m=1}^{M}$ interlaces over $\set{\kappa_m}_{m=1}^{M}$ we also have $\kappa_m\geq\lambda_{m+1}$.
If $\kappa_m>\lambda_{m+1}\geq\min\set{\lambda_m,\alpha_{m-p}}$ these facts imply $\lambda_{m+1}>\lambda_{m+1}$, a contradiction.
In particular, if $\kappa_m>\lambda_{m+1}$ we necessarily have $\lambda_{m+1}<\min\set{\lambda_m,\alpha_{m-p}}$ at which point~\eqref{equation.proof of Chop Kill step 11} gives $\kappa_m\leq\min\set{\lambda_m,\alpha_{m-p}}$.
Thus, for any $m=1,\dotsc,M$ we either have $\kappa_m=\lambda_{m+1}$ or $\lambda_{m+1}<\kappa_m\leq\min\set{\lambda_m,\alpha_{m-p}}$.
Moreover, for any $m=k,\dotsc,M$ the fact that $p+1\leq k$ gives $m-k+1\leq m-p$ and so $\alpha_{m-k+1}\geq\alpha_{m-p}\geq\min\set{\lambda_m,\alpha_{m-p}}$.
As such, for any $m=k,\dotsc,M$ we either have the intervals $[0,\kappa_m)$ and $[0,\lambda_{m+1})$ are equal or that both are disjoint from the interval $[\alpha_{m-k+1},\alpha_{m-k})$, implying
\begin{equation}
\label{equation.proof of Chop Kill step 12}
\sum_{m=k}^{M}\ell\bigset{[0,\kappa_m)\cap[\alpha_{m-k+1},\alpha_{m-k})}
=\sum_{m=k}^{M}\ell\bigset{[0,\lambda_{m+1})\cap[\alpha_{m-k+1},\alpha_{m-k})},\quad p+1\leq k\leq M.
\end{equation}

With~\eqref{equation.proof of Chop Kill step 10} and~\eqref{equation.proof of Chop Kill step 12} in hand, we now consider~\eqref{equation.proof of Chop Kill step 8} in the cases where $\set{\gamma_m}_{m=1}^{M}$ is $\set{\kappa_m}_{m=1}^{M}$ and $\set{\lambda_{m+1}}_{m=1}^{M}$, respectively.
In particular, for any $j$ such that $p+1\leq j\leq M$ note that $k\geq p+1$ for all $k\geq j$.
As such, in this case we can let $\gamma_m=\kappa_m$ in~\eqref{equation.proof of Chop Kill step 8} and apply~\eqref{equation.proof of Chop Kill step 12} for every $k$:
\begin{equation*}
\sum_{m=j}^{M}(\kappa_m-\alpha_{m-j+1})^+
=\sum_{k=j}^{M}\sum_{m=k}^{M}\ell\bigset{[0,\kappa_m)\cap[\alpha_{m-k+1},\alpha_{m-k})}
=\sum_{k=j}^{M}\sum_{m=k}^{M}\ell\bigset{[0,\lambda_{m+1})\cap[\alpha_{m-k+1},\alpha_{m-k})},\quad p+1\leq j\leq M.
\end{equation*}
To further simplify this expression we let $\gamma_m=\lambda_{m+1}$ in~\eqref{equation.proof of Chop Kill step 8}, recall that $\lambda_{M+1}:=0$, and replace ``$m$" with $m-1$:
\begin{equation}
\label{equation.proof of Chop Kill step 13}
\sum_{m=j}^{M}(\kappa_m-\alpha_{m-j+1})^+
=\sum_{m=j}^{M}(\lambda_{m+1}-\alpha_{m-j+1})^+
=\sum_{m=j}^{M-1}(\lambda_{m+1}-\alpha_{m-j+1})^+
=\sum_{m=j+1}^{M}(\lambda_m-\alpha_{m-j})^+,\quad p+1\leq j\leq M.
\end{equation}
Independent from this line of reasoning, note that replacing ``$j$" with $j+1$ in our assumption~\eqref{equation.completion Schur-Horn} gives
\begin{equation}
\label{equation.proof of Chop Kill step 14}
\sum_{m=j+1}^{M}(\lambda_m-\alpha_{m-j})^+\leq\sum_{n=j+1}^{N}\mu_n,\quad 1\leq j+1\leq M.
\end{equation}
Moreover, \smash{$\sum_{n=j+1}^{N}\mu_n\leq\sum_{n=j}^{N-1}\mu_n$} for all $j\geq 1$: if $j+1>N$ the left-hand side is zero, while if $j+1\leq N$ the fact that $\set{\mu_n}_{n=1}^{N}$ is nonincreasing gives $\sum_{n=j}^{N-1}\mu_n=(\mu_j-\mu_N)+\sum_{n=j+1}^{N}\mu_n\geq \sum_{n=j+1}^{N}\mu_n$.
Combining this fact with~\eqref{equation.proof of Chop Kill step 13} and~\eqref{equation.proof of Chop Kill step 14} then gives our claimed inequality in~\eqref{equation.Chop Kill step 2} in the special case where $p+1\leq j\leq M-1$:
\begin{equation*}
\sum_{m=j}^{M}(\kappa_m-\alpha_{m-j+1})^+
\leq\sum_{n=j+1}^{N}\mu_n
\leq\sum_{n=j}^{N-1}\mu_n,
\quad p+1\leq j\leq M-1.
\end{equation*}
Furthermore, \eqref{equation.Chop Kill step 2} immediately holds if $p+1\leq j=M$ since in this case~\eqref{equation.proof of Chop Kill step 13} gives $\sum_{m=M}^{M}(\kappa_m-\alpha_{m-M+1})^+=0$.

To summarize, we are in the process of showing that the inequality in~\eqref{equation.Chop Kill step 2} holds for all $j=1,\dotsc,M$ and so far, we have shown that it indeed does whenever $p+1\leq j\leq M$.
Since $1\leq p\leq M$ by assumption, what remains are the cases where $p=M$, $j=1,\dotsc,M$ and where $p+1\leq M$, $j<p+1$;
together these correspond to simply when $1\leq j\leq p$.
To prove the inequality in~\eqref{equation.Chop Kill step 2} holds for any $j=1,\dotsc,p$, we again let $\gamma_m=\kappa_m$ in~\eqref{equation.proof of Chop Kill step 8}:
\begin{equation}
\label{equation.proof of Chop Kill step 15}
\sum_{m=j}^{M}(\kappa_m-\alpha_{m-j+1})^+
=\sum_{k=j}^{M}\sum_{m=k}^{M}\ell\bigset{[0,\kappa_m)\cap[\alpha_{m-k+1},\alpha_{m-k})},
\quad 1\leq j\leq M.
\end{equation}
Note that in the $j=1$ case, the fact that $\kappa_m\geq\alpha_m$ along with the equality in~\eqref{equation.proof of Chop Kill step 7} gives
\begin{equation}
\label{equation.proof of Chop Kill step 16}
\sum_{n=1}^{N-1}\mu_n
=\sum_{m=1}^{M}(\kappa_m-\alpha_m)
=\sum_{m=1}^{M}(\kappa_m-\alpha_m)^+
=\sum_{k=1}^{M}\sum_{m=k}^{M}\ell\bigset{[0,\kappa_m)\cap[\alpha_{m-k+1},\alpha_{m-k})}.
\end{equation}
Subtracting~\eqref{equation.proof of Chop Kill step 15} from~\eqref{equation.proof of Chop Kill step 16} then gives
\begin{equation*}
\sum_{n=1}^{N-1}\mu_n-\sum_{m=j}^{M}(\kappa_m-\alpha_{m-j+1})^+=\sum_{k=1}^{j-1}\sum_{m=k}^{M}\ell\bigset{[0,\kappa_m)\cap[\alpha_{m-k+1},\alpha_{m-k})}, \quad 1\leq j\leq M.
\end{equation*}
In particular, for any $j=1,\dotsc,p$ we have $k\leq p-1$ whenever $1\leq k\leq j-1$ and so we may use~\eqref{equation.proof of Chop Kill step 10} to rewrite the right-hand side of the above equation:
\begin{equation}
\label{equation.proof of Chop Kill step 17}
\sum_{n=1}^{N-1}\mu_n-\sum_{m=j}^{M}(\kappa_m-\alpha_{m-j+1})^+=\sum_{k=1}^{j-1}\sum_{m=k}^{M}\ell\bigset{[0,\lambda_m)\cap[\alpha_{m-k+1},\alpha_{m-k})}, \quad 1\leq j\leq p.
\end{equation}
We now repeat this same process, starting with $\lambda_m$ instead of $\kappa_m$.
To be precise, subtracting~\eqref{equation.proof of Chop Kill step 8} from the $j=1$ case of itself, then letting $\gamma_m=\lambda_m$ and using the equality assumption of \eqref{equation.completion Schur-Horn} gives
\begin{equation}
\label{equation.proof of Chop Kill step 18}
\sum_{n=1}^{N-1}\mu_n-\sum_{m=j}^{M}(\lambda_m-\alpha_{m-j+1})^+
=\sum_{m=1}^{M}(\lambda_m-\alpha_{m-j+1})^+-\sum_{m=j}^{M}(\lambda_m-\alpha_{m-j+1})^+
=\sum_{k=1}^{j-1}\sum_{m=k}^{M}\ell\bigset{[0,\lambda_m)\cap[\alpha_{m-k+1},\alpha_{m-k})}
\end{equation}
for all $j=1,\dotsc,M$.
For any $j=1,\dotsc,p$, equating~\eqref{equation.proof of Chop Kill step 17} and~\eqref{equation.proof of Chop Kill step 18} and simplifying then gives
\begin{equation*}
\sum_{m=j}^{M}(\kappa_m-\alpha_{m-j+1})^+
=\sum_{m=j}^{M}(\lambda_m-\alpha_{m-j+1})^+-\mu_N,
\quad 1\leq j\leq p,
\end{equation*}
at which point, our assumption~\eqref{equation.completion Schur-Horn} gives the $j$th desired inequality of~\eqref{equation.Chop Kill step 2} in the remaining case where $j=1,\dotsc,p$:
\begin{equation*}
\sum_{m=j}^{M}(\kappa_m-\alpha_{m-j+1})^+
=\sum_{m=j}^{M}(\lambda_m-\alpha_{m-j+1})^+-\mu_N
\leq\sum_{n=j}^{N}\mu_n-\mu_N
=\sum_{n=j}^{N-1}\mu_n,
\quad 1\leq j\leq p.
\end{equation*}
Though obvious in the case where $j\leq N$, the final equality above has a subtle justification in the case where $j>N$:
here we have $0\leq \sum_{m=j}^{M}(\kappa_m-\alpha_{m-j+1})^+\leq-\mu_N$ which requires $\mu_N=0$, implying $\sum_{n=j}^{N}\mu_n-\mu_N=0-0=0=\sum_{n=j}^{N-1}\mu_n$.
\end{proof}

We now use Lemmas~\ref{lemma.eigensteps}, \ref{lemma.chopped spectra} and~\ref{lemma.Chop Kill step} to prove the ``if" direction of Theorem~\ref{theorem.main result 1}.
\begin{proof}[Proof of the ($\Leftarrow$) direction of Theorem~\ref{theorem.main result 1}]
Assume $\bfalpha=\set{\alpha_m}_{m=1}^{M}$, $\bflambda=\set{\lambda_m}_{m=1}^{M}$ and $\bfmu=\set{\mu_n}_{n=1}^{N}$ are nonnegative nonincreasing sequences with $\alpha_m\leq\lambda_m$ for all $m$ which satisfy~\eqref{equation.completion Schur-Horn}.
To show that $\bflambda$ is an $(\bfalpha,\bfmu)$-completion, it suffices by Lemma~\ref{lemma.eigensteps} to construct a sequence of eigensteps $\set{\set{\lambda_{P;m}}_{m=1}^{M}}_{P=0}^{N}$ from $\bfalpha$ to $\bflambda$ with lengths $\bfmu$, cf.\ Definition~\ref{definition.eigensteps}.
We construct these eigensteps iteratively: let $\lambda_{N;m}:=\lambda_m$ for all $m$ as required by condition (i) of Definition~\ref{definition.eigensteps}; for any given $P=1,\dotsc,N$, apply Lemma~\ref{lemma.chopped spectra} with
``$N$", ``$\set{\lambda_m}_{m=1}^{M}$" and ``$\set{\mu_n}_{n=1}^{N}$" being $P$, $\set{\lambda_{P;m}}_{m=1}^{M}$ and $\set{\mu_n}_{n=1}^{P}$, respectively, and define $\set{\lambda_{P-1;m}}_{m=1}^{M}$  to be the resulting sequence $\set{\kappa_m}_{m=1}^{M}$.
This construction is well-defined.
Indeed, our assumption~\eqref{equation.completion Schur-Horn} means that $\set{\lambda_{N;m}}_{m=1}^{M}=\set{\lambda_m}_{m=1}^{M}$ and $\set{\mu_n}_{n=1}^{N}$ satisfy the hypotheses of Lemma~\ref{lemma.chopped spectra}.
Moreover, for any given $P=1,\dotsc,N$, if $\set{\lambda_{P;m}}_{m=1}^{M}$ and $\set{\mu_n}_{n=1}^{P}$ satisfy the hypotheses of Lemma~\ref{lemma.chopped spectra}, then Lemma~\ref{lemma.Chop Kill step} guarantees that $\set{\lambda_{P-1;m}}_{m=1}^{M}$ and $\set{\mu_n}_{n=1}^{P-1}$ also satisfy these same hypotheses.
In particular, we necessarily have $\sum_{m=1}^{M}(\lambda_{P;m}-\alpha_m)=\sum_{n=1}^{P}\mu_n$ for all $P=1,\dotsc,N$, meaning condition (iii) of Definition~\ref{definition.eigensteps} holds for such $P$.
Further note that in the $P=1$ case, Lemmas~\ref{lemma.chopped spectra} and~\ref{lemma.Chop Kill step} imply $\set{\lambda_{0;m}}_{m=1}^{M}$ can and must be defined as $\set{\alpha_m}_{m=1}^{M}$ meaning we have satisfied both condition (i) as well as the $P=0$ case of condition (iii).
Finally, for any $P=1,\dotsc,N$, Lemma~\ref{lemma.Chop Kill step} guarantees that $\set{\lambda_{P;m}}_{m=1}^{M}$ interlaces over $\set{\lambda_{P-1;m}}_{m=1}^{M}$, namely condition (iv).
\end{proof}

We conclude this section with a brief discussion of how we should combine the above arguments with those in the existing literature in order to explicitly compute the actual vectors $\set{\bfphi_n}_{n=1}^{N}$ of an $(\bfalpha,\bfmu)$-completion of a given positive semidefinite operator $\bfA$.
To be clear, this process requires an explicit knowledge of the eigenvalues $\set{\alpha_m}_{m=1}^{M}$ of $\bfA$ as well as their corresponding eigenvectors.
It does not depend on the particular initial vectors whose frame operator is $\bfA$, nor is that information useful to this process.

Given the initial spectrum $\set{\alpha_m}_{m=1}^{M}$ as well as the sequence $\set{\mu_n}_{n=1}^{N}$ of desired squared-lengths, the first step is to determine the spectrum $\set{\lambda_m}_{m=1}^{M}$ that we wish to achieve in the completion $\bfA+\sum_{n=1}^{N}\bfphi_n^{}\bfphi_n^*$.
As we have just finished showing, $\set{\lambda_m}_{m=1}^{M}$ can be any nonnegative nonincreasing sequence that satisfies~\eqref{equation.completion Schur-Horn} with $\lambda_m\geq\alpha_m$ for all $m$.
A natural choice for $\set{\lambda_m}_{m=1}^{M}$ is the optimal such spectrum; as shown in the next section, this can be computed using the algorithm of Theorem~\ref{theorem.main result 2}.
Once such a spectrum $\set{\lambda_m}_{m=1}^{M}$ has been chosen, the next step is to form a sequence of eigensteps $\set{\set{\lambda_{P;m}}_{m=1}^{M}}_{P=0}^{N}$ from $\set{\alpha_m}_{m=1}^{M}$ to $\set{\lambda_m}_{m=1}^{M}$.
There may be many different ways to do this.
It is not hard to see that the set of all such sequences of eigensteps forms a convex polytope in $\bbR^{M(N+1)}$.
However, to date, an explicit parametrization of this polytope has only been found in the special case where $\alpha_m=0$ for all $m$~\cite{FickusMPS13}.
Nevertheless, we do now know that one such sequence always exists: by Lemmas~\ref{lemma.chopped spectra} and~\ref{lemma.Chop Kill step}, we can form a suitable spectrum $\set{\lambda_{P-1;m}}_{m=1}^{M}$ by choosing it to have trace $\sum_{m=1}^{M}\alpha_m+\sum_{n=1}^{P-1}\mu_n$ and lie between two consecutive chopped spectra of $\set{\lambda_{P;m}}_{m=1}^{M}$.
Once the eigensteps $\set{\set{\lambda_{P;m}}_{m=1}^{M}}_{P=0}^{N}$ have been constructed, we then use them along with the techniques of~\cite{CahillFMPS13} to explicitly construct the completion's vectors $\set{\bfphi_n}_{n=1}^{N}$.
To do this, the best approach is to not go through the proof of Lemma~\ref{lemma.eigensteps} itself,
but rather verify that the arguments behind Theorems~2 and~7 of~\cite{CahillFMPS13} are still valid when the intial spectrum of zero is generalized to any nonnegative nonincreasing sequence $\set{\alpha_m}_{m=1}^{M}$.
To be precise, for any eigenvalue $\lambda\in\set{\lambda_{P-1;m}}_{m=1}^{M}$ of the operator $\bfA+\sum_{n=0}^{P-1}\bfphi_n^{}\bfphi_n^*$, the squared-norm of the component of $\bfphi_n$ that lies in the corresponding eigenspace is given by
\begin{equation*}
-\lim_{x\rightarrow\lambda}(x-\lambda)\frac{\prod_{m=1}^{M}(x-\lambda_{P;m})}{\prod_{m=1}^{M}(x-\lambda_{P-1;m})}.
\end{equation*}
The interested reader should see~\cite{Poteet12} for examples of this entire process.

\section{Constructing optimal completions: Proving Theorem~\ref{theorem.main result 2}}

In this section, we exploit the characterization of $(\bfalpha,\bfmu)$-completions given in Theorem~\ref{theorem.main result 1} to provide a simple recursive algorithm---explicitly given in Theorem~\ref{theorem.main result 2}---for computing the optimal such completion.
We begin with a brief motivation of the algorithm, then prove it indeed computes the optimal $(\bfalpha,\bfmu)$-completion, and conclude with a low-dimensional example of its application.

Our algorithm is recursive.
It computes the optimal completion $\set{\beta_m}_{m=1}^{M}$ by computing $\beta_M$, then $\beta_{M-1}$, then $\beta_{M-2}$, etc.
Following the intuition behind~\cite{FickusMPS13} and the previous section, we visualize $\set{\beta_m}_{m=1}^{M}$ as a pyramid with its smallest blocks at the top,
each eigenvalue $\beta_k$ providing a foundation for the levels $\set{\beta_m}_{m=k+1}^{M}$ above it.
From this perspective, the goal of our algorithm is to build a pyramid that is as steep as possible.

To better understand our approach, assume for the moment that for any given $k=1,\dotsc,M$ we have already computed the parts of this pyramid that lie above level $k$, namely $\set{\beta_m}_{m=k+1}^{M}$.
To be clear, in the $k=M$ case, we make no assumptions.
For any $t\in\bbR$, we define the \textit{$k$th intermediate optimal spectrum} $\set{\gamma_{k;m}(t)}_{m=1}^{M}$ as
\begin{equation}
\label{equation.proof of main result 2.1}
\gamma_{k;m}(t):=\left\{\begin{array}{ll}\beta_m,&k+1\leq m\leq M,\\\max\set{\alpha_m,t},&1\leq m\leq k.\end{array}\right.
\end{equation}
Essentially, the top of $\set{\gamma_{k;m}(t)}_{m=1}^{M}$ corresponds to the parts of the optimal spectrum $\set{\beta_m}_{m=1}^{M}$ that we have already computed, whereas the bottom is obtained by \textit{water filling}, a technique borrowed from the theory of communications; turning our pyramid on its side, values of the initial spectrum $\set{\alpha_m}_{m=1}^{M}$ that lie below the ``water level" $t$ are subsumed by $t$, while $\alpha_m$'s that lie above it remain unchanged.
To compute $\beta_k$, we keep increasing this water level $t$ until we get to a point where increasing it any more would result in an invalid completion.
That is, we let $\beta_k$ be the largest value of $t$ for which the $k$th intermediate spectrum $\set{\gamma_{k;m}(t)}_{m=1}^{M}$ satisfies the first $k$ inequality constraints of Theorem~\ref{theorem.main result 1}:
\begin{equation}
\label{equation.proof of main result 2.2}
\beta_k=\max\Biggset{t\in\bbR: \sum_{m=j}^{M}(\gamma_{k;m}(t)-\alpha_{m-j+1})^+\leq\sum_{n=j}^{N}\mu_j,\ \forall j=1,\dotsc,k}.
\end{equation}
Note here that in the $k$ iterate we do not need to explicitly require $\set{\gamma_{k;m}(t)}_{m=1}^{M}$ to satisfy the last $M-k$ such constraints;
using some of the analysis given below in the proof of Theorem~\ref{theorem.main result 2}, the curious reader can verify that they are automatically satisfied, though we omit this work and remain completely rigorous.
To simplify this expression for $\beta_k$,
note that for any $j,k=1,\dotsc,M$ with $j\leq k$, \eqref{equation.proof of main result 2.1} allows us to rewrite the constraint functions in~\eqref{equation.proof of main result 2.2} as
\begin{equation*}
\sum_{m=j}^{k}(\gamma_m(t)-\alpha_{m-j+1})^++\sum_{m=k+1}^{M}(\gamma_m(t)-\alpha_{m-j+1})^+
=\sum_{m=j}^{k}\bigparen{\max\set{\alpha_m,t}-\alpha_{m-j+1}}^++\sum_{m=k+1}^{M}(\beta_m-\alpha_{m-j+1})^+.
\end{equation*}
To simplify the first of these two sums, note that for any $m=j,\dotsc,k$,
\begin{equation*}
\bigparen{\max\set{\alpha_m,t}-\alpha_{m-j+1}}^+
=\bigparen{\max\set{\alpha_m-\alpha_{m-j+1},t-\alpha_{m-j+1}}}^+
=\max\set{0,\alpha_m-\alpha_{m-j+1},t-\alpha_{m-j+1}}.
\end{equation*}
Since $\set{\alpha_m}_{m=1}^{M}$ is nonincreasing, $\alpha_m-\alpha_{m-j+1}\leq0$, meaning this further simplifies to
\begin{equation*}
(\max\set{\alpha_m,t}-\alpha_{m-j+1})^+
=\max\set{0,t-\alpha_{m-j+1}}
=(t-\alpha_{m-j+1})^+.
\end{equation*}
In summary, for any $j,k=1,\dotsc,M$ with $j\leq k$,
\begin{equation*}
\sum_{m=j}^{M}(\gamma_m(t)-\alpha_{m-j+1})^+
=\sum_{m=j}^{k}(t-\alpha_{m-j+1})^++\sum_{m=k+1}^{M}(\beta_m-\alpha_{m-j+1})^+.
\end{equation*}
Combining this observation with~\eqref{equation.proof of main result 2.2} leads to the ``official" definition of $\beta_k$ given in Theorem~\ref{theorem.main result 2}.
The proof of Theorem~\ref{theorem.main result 2} is complicated, and as such, we write two components of it as separate lemmas.
In the first lemma, we provide an alternative perspective on the algorithm of Theorem~\ref{theorem.main result 2}
which allows us to prove that the spectrum $\set{\beta_m}_{m=1}^{M}$ is a well-defined $(\bfalpha,\bfmu)$-completion,
and also lays the groundwork for our subsequent results.

\begin{lemma}
\label{lemma.first lemma for optimal completions}
Let $\bfalpha=\{\alpha_m\}_{m=1}^M$ and $\bfmu=\{\mu_{n}\}_{n=1}^N$ be nonnegative and nonincreasing with $M\leq N$.
For any $k=1,\dotsc,M$, assume we have already constructed $\set{\beta_m}_{m=k+1}^M$ according to the algorithm of Theorem~\ref{theorem.main result 2}.
For any $j=1,\dotsc,k$ let
\begin{equation}
\label{equation.proof of main result 2.3}
f_{k;j}(t):=\sum_{m=j}^{k}(t-\alpha_{m-j+1})^++\sum_{m=k+1}^{M}(\beta_m-\alpha_{m-j+1})^+,
\qquad
\nu_j:=\sum_{n=j}^{N}\mu_n.
\end{equation}
Letting $f_{k;j}^{-1}(-\infty,\nu_j]$ denote the preimage of the interval $(-\infty,\nu_j]$ under the function $f_{k;j}:\bbR\rightarrow\bbR$,
there exists $b_{k;j}\in\bbR$ such that \smash{$f_{k;j}^{-1}(-\infty,\nu_j]=(-\infty,b_{k;j}]$}.
Also, the number $\beta_k$ given by Theorem~\ref{theorem.main result 2} can be expressed as
\begin{equation}
\label{equation.proof of main result 2.4}
\beta_k
=\max\bigset{t\in\bbR: f_{k;j}(t)\leq\nu_j,\ \forall j=1,\dotsc,k}
=\max\Biggset{\bigcap_{j=1}^{k}f_{k;j}^{-1}(-\infty,\nu_j]}
=\min\set{b_{k;j}}_{j=1}^{k}.
\end{equation}
In particular, $\set{\beta_m}_{m=1}^{M}$ is a well-defined $(\bfalpha,\bfmu)$-completion.
Moreover, $f_{k;j}(\beta_{k+1})=f_{k+1;j}(\beta_{k+1})$ whenever $1\leq j\leq k<M$.
\end{lemma}

\begin{proof}
Our first step in proving that $\set{\beta_m}_{m=1}^{M}$ is the optimal $(\bfalpha,\bfmu)$-completion is to show that it is well-defined.
We prove this by induction.
In particular, for any $k=1,\dotsc,M$, we assume we have already constructed $\set{\beta_m}_{m=k+1}^M$ according to~\eqref{equation.proof of main result 2.4}, and show that the maximum that defines $\beta_k$ in~\eqref{equation.proof of main result 2.4} exists.
We take care to note that our argument will even be valid in the $k=M$ case; there, we make no assumptions whatsoever about $\set{\beta_m}_{m=1}^{M}$.
Having already constructed $\set{\beta_m}_{m=k+1}^M$, note that for any $j=1,\dotsc,k$, the corresponding $f_{k;j}$ function~\eqref{equation.proof of main result 2.3} is well-defined.

At this point, note that under this notation,
the expression for $\beta_k$ given in Theorem~\ref{theorem.main result 2} reduces to:
\begin{equation*}
\beta_k
=\max\bigset{t\in\bbR: f_{k;j}(t)\leq\nu_j,\ \forall j=1,\dotsc,k}
=\max\Biggset{\bigcap_{j=1}^{k}f_{k;j}^{-1}(-\infty,\nu_j]},
\end{equation*}
namely the first part of~\eqref{equation.proof of main result 2.4}.
To prove this set indeed has a maximum, we investigate the properties of the sets \smash{$\set{f_{k;j}^{-1}(-\infty,\nu_j]}_{j=1}^{k}$}.
Our first claim is that \smash{$f_{k;j}^{-1}(-\infty,\nu_j]$} is nonempty for any $j=1,\dotsc,k$.
That is, for any such $j$, we claim there exists some $t\in\bbR$ such that $f_{k;j}(t)\leq\mu_j$.
Indeed, when $k=M$ we can take $t=\alpha_M$: since $\set{\alpha_m}_{m=1}^{M}$ is nonincreasing and $\set{\mu_n}_{n=1}^{N}$ is nonnegative, \eqref{equation.proof of main result 2.3} gives \smash{$f_{M;j}(\alpha_M)=\sum_{m=j}^{M}(\alpha_M-\alpha_{m-j+1})^+=0\leq\nu_j$} for any $j=1,\dotsc,M$.
If on the other hand $k<M$, we can take $t=\beta_{k+1}$.
To see this, note that for any $j=1,\dotsc,k$, considering~\eqref{equation.proof of main result 2.3} when ``$k$" is $k+1$ gives
\begin{equation*}
f_{k;j}(\beta_{k+1})
=\sum_{m=j}^{k}(\beta_{k+1}-\alpha_{m-j+1})^++\sum_{m=k+1}^{M}(\beta_m-\alpha_{m-j+1})^+
=\sum_{m=j}^{k+1}(\beta_{k+1}-\alpha_{m-j+1})^++\sum_{m=k+2}^{M}(\beta_m-\alpha_{m-j+1})^+
=f_{k+1;j}(\beta_{k+1}),
\end{equation*}
as claimed in the statement of the lemma.
Looking at our inductive hypothesis~\eqref{equation.proof of main result 2.4} where ``$k$" is $k+1$, we see that $\beta_{k+1;j}$ is the maximum of the intersection of the sets $\set{f_{k+1;j}^{-1}(-\infty,\nu_j]}_{j=1}^{k+1}$.
In particular, it is a member of each of these sets, implying via the previous equation that $f_{k;j}(\beta_{k+1})=f_{k+1;j}(\beta_{k+1})\leq\nu_j$ for any $j=1,\dotsc,k$.
Thus, for any such $j$, $\beta_{k+1}\in f_{k;j}^{-1}(-\infty,\nu_j]$ and so $f_{k;j}^{-1}(-\infty,\nu_j]\neq\emptyset$ as claimed.

Having that $f_{k;j}^{-1}(-\infty,\nu_j]$ is nonempty for any $j=1,\dotsc,k$, we next note that for any such $j$ there exists $b_{k;j}\in\bbR$ such that \smash{$f_{k;j}^{-1}(-\infty,\nu_j]=(-\infty,b_{k;j}]$}.
Indeed, for any such $j$ the corresponding $f_{k;j}$ function~\eqref{equation.proof of main result 2.3} is clearly continuous, piecewise linear and nondecreasing with $\lim_{t\rightarrow\infty}f_{k;j}(t)=\infty$.
This last fact implies that $f_{k;j}^{-1}(-\infty,\nu_j]$ is bounded above which, coupled with its nonemptiness, implies its supremum $b_{k;j}$ exists.
Moreover, since $f_{k;j}$ is continuous this set is closed and this supremum is, in fact, a maximum.
At this point the monotonicity of $f_{k;j}$ implies that $f_{k;j}^{-1}(-\infty,\nu_j]=(-\infty,b_{k;j}]$.
Putting all of this together gives the rest of~\eqref{equation.proof of main result 2.4},
which among other things, ensures $\beta_k$ is well-defined:
\begin{equation*}
\beta_k
=\max\Biggset{\bigcap_{j=1}^{k}f_{k;j}^{-1}(-\infty,\nu_j]}
=\max\Biggset{\bigcap_{j=1}^{k}(-\infty,b_{k;j}]}
=\max\bigl(-\infty,\min\set{b_{k;j}}_{j=1}^{k}\bigr]
=\min\set{b_{k;j}}_{j=1}^{k}.
\end{equation*}
In particular, the iterative process given in the theorem statement will indeed produce a sequence $\set{\beta_m}_{m=1}^{M}$.
Moreover, recall from above that if $k<M$ then $\beta_{k+1}\in f_{k;j}^{-1}(-\infty,\nu_j]$ for all $j=1,\dotsc,k$.
Thus, \smash{$\beta_{k+1}\in\cap_{j=1}^{k}f_{k;j}^{-1}(-\infty,\nu_j]$} and so \smash{$\beta_{k+1}\leq\max\set{\cap_{j=1}^{k}f_{k;j}^{-1}(-\infty,\nu_j]}=\beta_k$}.
As such, \smash{$\set{\beta_m}_{m=1}^{M}$} is nonincreasing.

We now claim that $\set{\beta_m}_{m=1}^{M}$ is an $(\bfalpha,\bfmu)$-completion, namely that it satisfies $\alpha_m\leq\beta_m$ for all $m=1,\dotsc,M$ and moreover
the conditions~\eqref{equation.completion Schur-Horn} given in Theorem~\ref{theorem.main result 1}.
To show $\alpha_k\leq\beta_k$ for any $k=1,\dotsc,M$, recall that $\set{\alpha_m}_{m=1}^{M}$ is nonincreasing.
As such, for any $j=1,\dotsc,k$ and any $m=j,\dotsc,k$ we have $m+1\leq k+j$, implying $m-j+1\leq k$ and so $\alpha_{m-j+1}\geq\alpha_k$.
In particular, $(\alpha_k-\alpha_{m-j+1})^+=0\leq(\beta_k-\alpha_{m-j+1})^+$ for all such $m$ and so evaluating $f_{k;j}$~\eqref{equation.proof of main result 2.3} at $t=\alpha_k$ and $t=\beta_k$ gives
\begin{equation*}
f_{k;j}(\alpha_k)
=\sum_{m=j}^{k}(\alpha_k-\alpha_{m-j+1})^++\sum_{m=k+1}^{M}(\beta_m-\alpha_{m-j+1})^+
\leq\sum_{m=j}^{k}(\beta_k-\alpha_{m-j+1})^++\sum_{m=k+1}^{M}(\beta_m-\alpha_{m-j+1})^+
=f_{k;j}(\beta_k).
\end{equation*}
Moreover, recall from~\eqref{equation.proof of main result 2.4} that $\beta_k$ lies in the set \smash{$\cap_{j=1}^{k}f_{k;j}^{-1}(-\infty,\nu_j]$} being its maximum.
Thus, $f_{k;j}(\alpha_k)\leq f_{k;j}(\beta_k)\leq\nu_j$ for all $j=1,\dotsc,k$, meaning $\alpha_k$ also lies in \smash{$\cap_{j=1}^{k}f_{k;j}^{-1}(-\infty,\nu_j]$}, and is therefore no greater than its maximum.
That is, $\alpha_k\leq\beta_k$ for all $k=1,\dotsc,K$, as claimed.
Moreover, note that letting $j=k$ in the above discussion gives $f_{k;k}(\beta_k)\leq\nu_k$.
Considering~\eqref{equation.proof of main result 2.3} when $j=k$, this inequality becomes the $k$th necessary inequality of Theorem~\ref{theorem.main result 1}:
\begin{equation*}
\sum_{m=k}^{M}(\beta_m-\alpha_{m-k+1})^+
=\sum_{m=k}^{k}(\beta_k-\alpha_{m-k+1})^++\sum_{m=k+1}^{M}(\beta_m-\alpha_{m-k+1})^+
=f_{k;k}(\beta_k)
\leq\nu_k
=\sum_{n=k}^{N}\mu_n.
\end{equation*}
Finally, to prove that $\set{\beta_m}_{m=1}^{M}$ also satisfies the equality condition of Theorem~\ref{theorem.main result 1}, consider~\eqref{equation.proof of main result 2.4} when $k=1$, namely that $\beta_1$ is defined to be $\beta_1=\max\bigset{t\in\bbR: f_{1;1}(t)\leq\nu_1}$.
Being a member of this set, $\beta_1$ necessarily satisfies $f_{1;1}(\beta_1)\leq\nu_1$.
Moreover, if $f_{1;1}(\beta_1)<\nu_1$ then since $f_{1;1}$ is continuous, we would have $f_{1;1}(\beta_1+\varepsilon)<\nu_1$ for all sufficiently small $\varepsilon>0$, contradicting the definition of $\beta_1$.
Thus, $f_{1;1}(\beta_1)=\nu_1$ and considering~\eqref{equation.proof of main result 2.3} when $j=k=1$ gives our desired equality:
\begin{equation*}
\sum_{m=1}^M(\beta_m-\alpha_m)
=\sum_{m=1}^M(\beta_m-\alpha_m)^+
=(\beta_1-\alpha_1)^++\sum_{m=2}^{M}(\beta_m-\alpha_m)^+
=f_{1;1}(\beta_1)
=\nu_1
=\sum_{n=1}^{N}\mu_n.\qedhere
\end{equation*}
\end{proof}

Having that $\set{\beta_m}_{m=1}^{M}$ is a well-defined $(\bfalpha,\bfmu)$-completion, all that remains to be shown is that $\set{\beta_m}_{m=1}^{M}$ is minimal.
That is, letting $\set{\lambda_m}_{m=1}^{M}$ be any $(\bfalpha,\bfmu)$-completion we show that $\set{\beta_m}_{m=1}^{M}\preceq\set{\lambda_m}_{m=1}^{M}$.
Since both sequences sum to $\sum_{m=1}^{M}\alpha_m+\sum_{n=1}^{N}\mu_n$ by definition, this reduces to demonstrating that
\begin{equation}
\label{equation.proof of main result 2.8}
\sum_{m=j}^{M}\lambda_m\leq\sum_{m=j}^{M}\beta_m,\quad\forall j=1,\dotsc,M.
\end{equation}

Before proving~\eqref{equation.proof of main result 2.8} itself, we first develop a better understanding of $\set{\beta_m}_{m=1}^{M}$.
For any given $k=1,\dotsc,M$, recall from earlier in this proof that for any $j=1,\dotsc,k$, there exists $b_{j,k}\in\bbR$ such that \smash{$f_{k;j}^{-1}(-\infty,\nu_j]=(-\infty,b_{k;j}]$}.
This led to~\eqref{equation.proof of main result 2.4}, namely that $\beta_k=\min\set{b_{k;j}}_{j=1}^{k}$.
Some members of the sequence $\set{b_{k;j}}_{j=1}^{k}$ will equal this minimum, while others will not; in the following result, we prove some special properties of the smallest index $j$ that does.

\begin{lemma}
\label{lemma.second lemma for optimal completions}
Following the same hypotheses and notation as Lemma~\ref{lemma.first lemma for optimal completions}, let
\begin{equation}
\label{equation.proof of main result 2.9}
j(k):=\min\calJ(k),
\qquad
\calJ(k):=\set{j=1,\dotsc,k: b_{k;j}=\beta_k}=\Bigset{j=1,\dotsc,k: \max\bigset{f_{k;j}^{-1}(-\infty,\nu_j]}=\beta_k}.
\end{equation}
The set $\calJ(k)$ and index $j(k)$ have the following three properties:
\begin{enumerate}
\renewcommand{\labelenumi}{(\alph{enumi})}
\item
$f_{k;j}(\beta_k)=\nu_j$ for all $j\in\calJ(k)$.
\item
$\alpha_{k-j(k)+1}\leq\beta_k$ for all $k=1,\dotsc,M$.
\item
$j(k)\leq j(k+1)$ for all $k=1,\dotsc,M-1$.
\end{enumerate}
\end{lemma}

\begin{proof}
From~\eqref{equation.proof of main result 2.4}, note that $\beta_k$ is the largest value of $t$ for which $f_{k;j}(t)\leq \nu_j$ for all $j=1,\dotsc,k$, namely for which the $k$th intermediate spectrum~\eqref{equation.proof of main result 2.1} will satisfy the first $k$ inequality conditions of $(\bfalpha,\bfmu)$-completions given in Theorem~\ref{theorem.main result 1}.
That is, $\calJ(k)$ consists of those indices $j$ for which even slightly increasing $t$ beyond $\beta_k$ will violate $f_{k;j}(t)\leq \nu_j$.
Indeed, for any $j=1,\dotsc,k$ we have $j\in\calJ(k)$ if and only if \smash{$f_{k;j}^{-1}(-\infty,\nu_j]=(-\infty,\beta_k]$};
since preimages preserve set complements this happens precisely when \smash{$f_{k;j}^{-1}(\nu_j,\infty)=(\beta_k,\infty)$}, meaning~\eqref{equation.proof of main result 2.9} can be equivalently expressed as
\begin{equation}
\label{equation.proof of main result 2.10}
j(k)=\min\calJ(k),
\qquad
\calJ(k)=\set{j=1,\dotsc,k: f_{k;j}(t)>\nu_j,\ \forall t>\beta_k}.
\end{equation}
Note that for any $j\in\calJ(k)$, \eqref{equation.proof of main result 2.9} gives $f_{k;j}(\beta_k)\leq\nu_{j}$ while~\eqref{equation.proof of main result 2.10} gives $f_{k;j}(t)>\nu_{j}$ for all $t>\beta_k$.
Since each $f_{k;j}$ is continuous, this implies $f_{k;j}(\beta_k)=\nu_j$ for all such $j$, namely (a).

We next prove (b).
This claim can be viewed as a strengthening of the $\alpha_k\leq\beta_k$ inequality we proved earlier.
To prove it, recall that for any $k=1,\dotsc,M$ we have $\beta_k\geq\beta_M\geq\alpha_M$.
Since $\set{\alpha_m}_{m=1}^{M}$ is nonincreasing, there thus exists a unique index $m(k)$ such that $1\leq m(k)\leq M$ and such that $\alpha_{m(k)}\leq\beta_k<\alpha_{m(k)-1}$,
provided we adopt the convention of defining $\alpha_0:=\infty$.
To prove~(b), we first show that $k$, $j(k)$ and $m(k)$ are all related by the following inequality:
\begin{equation}
\label{equation.proof of main result 2.13}
j(k)\leq k-m(k)+1,\quad\forall k=1,\dotsc,M.
\end{equation}
Note that since $j(k)\leq k$ by definition~\eqref{equation.proof of main result 2.9}, it suffices to consider the case where $m(k)\geq 2$.
Assume to the contrary that $k-m(k)+1<j(k)$,
and note that for all $m=j(k),\dotsc,k$ we have $m-m(k)+2\leq k-m(k)+2\leq j(k)$, implying $m-j(k)+1\leq m(k)-1$ and so $\alpha_{m(k)-1}\leq\alpha_{m-j(k)+1}$.
In particular, for all $m=j(k),\dotsc,k$ we have $(t-\alpha_{m-j(k)+1})^+=0$ for all $t\leq\alpha_{m(k)-1}$.
Thus, considering~\eqref{equation.proof of main result 2.3} at $j=j(k)$, we see that for any $t\leq\alpha_{m(k)-1}$,
\begin{equation*}
f_{k;j(k)}(t)
=\sum_{m=j(k)}^{k}(t-\alpha_{m-j(k)+1})^++\sum_{m=k+1}^{M}(\beta_m-\alpha_{m-j+1})^+
=\sum_{m=k+1}^{M}(\beta_m-\alpha_{m-j+1})^+.
\end{equation*}
That is, the function $f_{k;j(k)}$ is necessarily constant over all $t\leq \alpha_{m(k)-1}$.
Since this includes $\beta_k$ by the definition of $m(k)$, we have $f_{k;j(k)}(\alpha_{m(k)-1})=f_{k;j(k)}(\beta_k)\leq\nu_{j(k)}$,
meaning \smash{$\alpha_{m(k)-1}\in f_{k;j(k)}^{-1}(-\infty,\mu_{j(k)}]$}.
But by~\eqref{equation.proof of main result 2.9}, $j(k)\in\calJ(k)$  meaning
\smash{$\beta_k=\max\bigset{f_{k;j(k)}^{-1}(-\infty,\nu_j(k)]}\geq\alpha_{m(k)-1}$}, a contradiction of the fact that $\beta_k<\alpha_{m(k)-1}$.
Thus~\eqref{equation.proof of main result 2.13} is indeed true.
Rewriting~\eqref{equation.proof of main result 2.13} as $m(k)\leq k-j(k)+1$, claim (b) follows immediately from the definition of $m(k)$ and the fact that $\set{\alpha_m}_{m=1}^{M}$ is nonincreasing:
$\alpha_{k-j(k)+1}\leq\alpha_{m(k)}\leq\beta_k$.

Finally, we prove (c).
Our argument relies on a more basic fact, namely that $f_{k;i}-f_{k;j}$ is nondecreasing for any $k=1,\dotsc,M$ and any $i\leq j\leq k$.
Indeed, for any such $i$, $j$ and $k$, \eqref{equation.proof of main result 2.3} gives
\begin{align*}
f_{k;i}(t)-f_{k;j}(t)
&=\sum_{m=i}^{k}(t-\alpha_{m-i+1})^++\sum_{m=k+1}^{M}(\beta_m-\alpha_{m-i+1})^+-\sum_{m=j}^{k}(t-\alpha_{m-j+1})^+-\sum_{m=k+1}^{M}(\beta_m-\alpha_{m-j+1})^+\\
&=\sum_{m=i}^{j-1}(t-\alpha_{m-i+1})^++\sum_{m=j}^{k}[(t-\alpha_{m-i+1})^+-(t-\alpha_{m-j+1})^+]+\sum_{m=k+1}^{M}[(\beta_m-\alpha_{m-i+1})^+-(\beta_m-\alpha_{m-j+1})^+],
\end{align*}
where all summands are nondecreasing: the summands of the first and third sum are clearly nondecreasing and, since $i\leq j$ implies $\alpha_{m-i+1}\leq\alpha_{m-j+1}$, the summands of the second sum, namely
\begin{equation*}
(t-\alpha_{m-i+1})^+-(t-\alpha_{m-j+1})^+
=\left\{\begin{array}{ll}
0,&t\leq\alpha_{m-i+1},\\
t-\alpha_{m-i+1},&\alpha_{m-i+1}\leq t\leq\alpha_{m-j+1},\\
\alpha_{m-j+1}-\alpha_{m-i+1},&\alpha_{m-j+1}\leq t,
\end{array}\right.
\end{equation*}
are nondecreasing as well.
Returning to the claim (c) that $j(k)\leq j(k+1)$ for any $k=1,\dotsc,M-1$, assume to the contrary that $j(k+1)<j(k)$, implying $f_{k;j(k+1)}-f_{k;j(k)}$ is nondecreasing.
In particular, for any $t>\beta_k$ we can evaluate $f_{k;j(k+1)}-f_{k;j(k)}$ at $\beta_k$ and $t$ to obtain
$f_{k;j(k+1)}(\beta_k)-f_{k;j(k)}(\beta_k)\leq f_{k;j(k+1)}(t)-f_{k;j(k)}(t)$ or equivalently, that
\begin{equation*}
f_{k;j(k+1)}(\beta_k)+f_{k;j(k)}(t)\leq f_{k;j(k+1)}(t)+f_{k;j(k)}(\beta_k),\quad\forall t>\beta_k.
\end{equation*}
At this point, recall that since $j(k)\in\calJ(k)$, (a) gives $f_{k;j(k)}(\beta_k)=\nu_{j(k)}$ while~\eqref{equation.proof of main result 2.10} gives $f_{k;j(k)}(t)>\nu_{j(k)}$ for all $t>\beta_k$.
Thus, the previous inequality implies that
\begin{equation*}
f_{k;j(k+1)}(\beta_k)+\nu_{j(k)}
< f_{k;j(k+1)}(\beta_k)+f_{k;j(k)}(t)
\leq f_{k;j(k+1)}(t)+f_{k;j(k)}(\beta_k)
=f_{k;j(k+1)}(t)+\nu_{j(k)},
\quad\forall t>\beta_k,
\end{equation*}
namely that $f_{k;j(k+1)}(\beta_k)<f_{k;j(k+1)}(t)$ for all $t>\beta_k$.
Moreover, since $f_{k;j(k+1)}$ is a nondecreasing function and $\beta_{k+1}\leq\beta_k$ we know $f_{k;j(k+1)}(\beta_{k+1})\leq f_{k;j(k+1)}(\beta_k)$.
Also, since $j(k+1)<j(k)\leq k$ we can let ``$j$" be $j(k+1)$ in the final conclusion of Lemma~\ref{lemma.first lemma for optimal completions} to obtain $f_{k;j(k+1)}(\beta_{k+1})=f_{k+1;j(k+1)}(\beta_{k+1})$.
And, since $j(k+1)\in\calJ(k+1)$, (a) gives $f_{k+1;j(k+1)}(\beta_{k+1})=\mu_{j(k+1)}$.
Putting this all together, we see that
\begin{equation*}
\mu_{j(k+1)}
=f_{k+1;j(k+1)}(\beta_{k+1})
=f_{k;j(k+1)}(\beta_{k+1})
\leq f_{k;j(k+1)}(\beta_k)
<f_{k;j(k+1)}(t),
\quad\forall t>\beta_k.
\end{equation*}
Since $f_{k;j(k+1)}(t)>\mu_{j(k+1)}$ for all $t>\beta_k$, \eqref{equation.proof of main result 2.10} gives $j(k+1)\in\calJ(k)$ and so $j(k+1)\geq\min\calJ(k)=j(k)$, a contradiction of the assumption that $j(k+1)<j(k)$.
\end{proof}

Having Lemmas~\ref{lemma.first lemma for optimal completions} and~\ref{lemma.second lemma for optimal completions}, we prove our second main result:

\begin{proof}[Proof of Theorem~\ref{theorem.main result 2}]
Recall from Lemma~\ref{lemma.first lemma for optimal completions} that the algorithm of Theorem~\ref{theorem.main result 2} produces a well-defined $(\bfalpha,\bfmu)$-completion $\set{\beta_m}_{m=1}^{M}$.
As noted above, all that remains to be shown is that $\set{\beta_m}_{m=1}^{M}\preceq\set{\lambda_m}_{m=1}^{M}$ for any $(\bfalpha,\bfmu)$-completion $\set{\lambda_m}_{m=1}^{M}$, namely \eqref{equation.proof of main result 2.8}.
In light of the iterative definition of $\set{\beta_m}_{m=1}^{M}$, we prove~\eqref{equation.proof of main result 2.8} by induction, beginning with $j=M$ and working backwards to $j=1$.
In particular, for any $k=1,\dotsc,M$, assume we have already shown~\eqref{equation.proof of main result 2.8} holds whenever $k+1\leq j\leq M$; we show that it also holds for $j=k$.
As with our inductive argument for Lemma~\ref{lemma.first lemma for optimal completions}, our techniques below will even be valid in the $j=M$ case; in that case, we assume nothing about the optimality of $\set{\beta_m}_{m=1}^{M}$.

Note that if $\lambda_k\leq\beta_k$, the case of~\eqref{equation.proof of main result 2.8} with $j=k$ immediately follows from the $j=k+1$ case:
\begin{equation*}
\sum_{m=k}^{M}\lambda_m
=\lambda_k+\sum_{m=k+1}^{M}\lambda_m
\leq\beta_k+\sum_{m=k+1}^{M}\beta_m
=\sum_{m=k}^{M}\beta_m.
\end{equation*}
As such, assume $\lambda_k>\beta_k$.
Since $\set{\lambda_m}_{m=1}^{M}$ is an $(\bfalpha,\bfmu)$-completion, Theorem~\ref{theorem.main result 1} and~\eqref{equation.proof of main result 2.3} imply
\begin{equation*}
\sum_{m=j}^{M}(\lambda_m-\alpha_{m-j+1})^+
\leq\sum_{n=j}^{N}\mu_n
=\nu_j,
\end{equation*}
for any $j=1,\dotsc,M$.
Consider this inequality in the case where $j$ is the index $j(k)$ given in~\eqref{equation.proof of main result 2.9}.
In this case, recall that since $j(k)\in\calJ(k)$, Lemma~\ref{lemma.second lemma for optimal completions}(a) gives $f_{k;j(k)}(\beta_k)=\nu_{j(k)}$.
Putting these facts together with the explicit formula~\eqref{equation.proof of main result 2.3} for $f_{k;j(k)}(\beta_k)$ gives
\begin{equation*}
\sum_{m=j(k)}^{M}(\lambda_m-\alpha_{m-j(k)+1})^+
\leq\nu_{j(k)}
=f_{k;j(k)}(\beta_k)
=\sum_{m=j(k)}^{k}(\beta_k-\alpha_{m-j(k)+1})^++\sum_{m=k+1}^{M}(\beta_m-\alpha_{m-j(k)+1})^+.
\end{equation*}
Rewriting the right-hand side above by grouping the $m=k$ term with the second sum instead of the first gives
\begin{equation*}
\sum_{m=j(k)}^{M}(\lambda_m-\alpha_{m-j(k)+1})^+
=\sum_{m=j(k)}^{k-1}(\beta_k-\alpha_{m-j(k)+1})^++\sum_{m=k}^{M}(\beta_m-\alpha_{m-j(k)+1})^+.
\end{equation*}
To continue, note that since $\set{\lambda_m}_{m=1}^{M}$ is nonincreasing, $\beta_k<\lambda_k\leq\lambda_m$ for all $m=1,\dotsc,k$.
In particular, for any $m$ such that $j(k)\leq m\leq k-1$ we know $(\beta_k-\alpha_{m-j(k)+1})^+\leq(\lambda_m-\alpha_{m-j(k)+1})^+$ and so the previous equality implies
\begin{equation*}
\sum_{m=j(k)}^{M}(\lambda_m-\alpha_{m-j(k)+1})^+
\leq\sum_{m=j(k)}^{k-1}(\lambda_m-\alpha_{m-j(k)+1})^++\sum_{m=k}^{M}(\beta_m-\alpha_{m-j(k)+1})^+.
\end{equation*}
Subtracting common terms from both sides of this inequality and then noting $x\leq x^+$ for all $x\in\bbR$ gives
\begin{equation}
\label{equation.proof of main result 2.16}
\sum_{m=k}^{M}(\lambda_m-\alpha_{m-j(k)+1})
\leq\sum_{m=k}^{M}(\lambda_m-\alpha_{m-j(k)+1})^+
\leq\sum_{m=k}^{M}(\beta_m-\alpha_{m-j(k)+1})^+.
\end{equation}
To continue, recall from Lemma~\ref{lemma.second lemma for optimal completions}(b) that $\alpha_{m-j(m)+1}\leq \beta_m$ for all $m=1,\dotsc,M$.
Further recalling that $\set{j(k)}_{k=1}^{M}$ is nondecreasing, for any $m=k,\dotsc,M$ we have $j(k)\leq j(m)$ implying $m-j(m)+1\geq m-j(k)+1$ and so $\alpha_{m-j(k)+1}\leq\alpha_{m-j(m)+1}$.
Together, these facts about $\set{j(k)}_{k=1}^{M}$ imply $\alpha_{m-j(k)+1}\leq\beta_m$ for all $m=k,\dotsc,M$, implying \eqref{equation.proof of main result 2.16} can be further simplified as
\begin{equation*}
\sum_{m=k}^{M}(\lambda_m-\alpha_{m-j(k)+1})
\leq\sum_{m=k}^{M}(\beta_m-\alpha_{m-j(k)+1})^+
=\sum_{m=k}^{M}(\beta_m-\alpha_{m-j(k)+1}).
\end{equation*}
Subtracting common terms from both sides gives that the inductive hypothesis is also true at $j=k$:
\begin{equation}
\label{equation.proof of main result 2.17}
\sum_{m=k}^{M}\lambda_m\leq\sum_{m=k}^{M}\beta_m.
\end{equation}
Thus, \eqref{equation.proof of main result 2.8} indeed holds for all $k=1,\dotsc,M$, meaning $\set{\beta_m}_{m=1}^{M}\preceq\set{\lambda_m}_{m=1}^{M}$ for any $(\bfalpha,\bfmu)$-completion $\set{\lambda_m}_{m=1}^{M}$.
To be clear, in the initial case where $j=M$, the above inductive argument assumes nothing about $\set{\beta_m}_{m=1}^{M}$.
In this case, it shows that if $\lambda_M>\beta_M$ then~\eqref{equation.proof of main result 2.17} holds for $k=M$, namely that $\lambda_M\leq\beta_M$.
As such, in the initial case, this argument reduces to a proof by contradiction that $\lambda_M\leq\beta_M$.
\end{proof}

To highlight the utility of Theorem~\ref{theorem.main result 2}, we now use it to compute an example of an optimal completion.
\begin{example}
\label{example}
Consider a $4\times 4$ self-adjoint matrix $\bfA$ whose spectrum is
\begin{equation*}
\bfalpha=\set{\alpha_1,\alpha_2,\alpha_3,\alpha_4}=\set{\tfrac{7}{4},\tfrac{3}{4},\tfrac{1}{2},\tfrac{1}{2}}.
\end{equation*}
From~\cite{FickusMPS13}, we know that $\bfA$ is the frame operator for infinitely many frames for $\bbR^4$ or $\bbC^4$ consisting of $4$ or more frame vectors.
Regardless of what particular frame led to $\bfA$, suppose we can add any $N=5$ additional vectors to this frame, the only restriction being that they have squared-norms of
\begin{equation*}
\bfmu=\set{\mu_{1},\mu_{2},\mu_{3},\mu_{4},\mu_{5}}=\set{2,1,\tfrac14,\tfrac14,\tfrac14}.
\end{equation*}
How should we pick these vectors so that the resulting frame is as tight as possible?  Or so that it has minimal frame potential, or alternatively, minimal mean squared reconstruction error?
As discussed in the introduction, Theorem~\ref{theorem.main result 2} shows that all of these questions have the same answer;
we explicitly construct an $(\bfalpha,\bfmu)$-completion $\set{\beta_m}_{m=1}^{M}$ that is majorized by all other $(\bfalpha,\bfmu)$-completions.
To be precise, for any $k=1,\dotsc,M$ we compute $\beta_k$ from $\set{\beta_m}_{m=k+1}^{M}$ by defining \smash{$f_{k;j}(t)=\sum_{m=j}^{k}(t-\alpha_{m-j+1})^++\sum_{m=k+1}^{4}(\beta_m-\alpha_{m-j+1})^+$} for all $j=1,\dotsc,k$ and $t\in\bbR$ and letting
$\beta_k:=\min\set{t : f_{k;j}(t)\leq \sum_{n=j}^{5}\mu_n,\ \forall j=1,\dotsc, k}$.
In particular, $\beta_4$ is the largest value of $t$ that satisfies the four constraints:
\begin{align*}
f_{4;1}(t)&=(t-\tfrac74)^++(t-\tfrac34)^++(t-\tfrac12)^++(t-\tfrac12)^+\leq\tfrac{15}{4},\\
f_{4;2}(t)&=(t-\tfrac74)^++(t-\tfrac34)^++(t-\tfrac12)^+\leq\tfrac{7}{4},\\
f_{4;3}(t)&=(t-\tfrac74)^++(t-\tfrac34)^+\leq\tfrac{3}{4},\\
f_{4;4}(t)&=(t-\tfrac74)^+\leq\tfrac{1}{2}.
\end{align*}
Here, each of the constraints can be explicitly written in terms of a piecewise linear function.
For example,
\begin{equation*}
f_{4;1}(t)=\left\{\begin{array}{cl}0,& \multicolumn{1}{r}{t<\frac12,}\smallskip\\2t-1,&\frac12\leq t<\frac34,\smallskip\\3t-\frac74 &\frac34\leq t<\frac74,\smallskip\\4t-\frac72, &\frac74\leq t,\end{array}\right.
\end{equation*}
at which point basic arithmetic reveals that the interval $(-\infty, \frac{29}{16}]$ is the set of points $t$ such that $f_{4;1}(t)\leq\frac{15}{4}$.
Similarly, the second, third and fourth constraints above correspond to the intervals $(-\infty, \frac32]$, $(-\infty, \frac32]$, and $(-\infty, \frac94]$, respectively.
The largest point that lies in all four intervals is $\beta_4:=\frac32$.
Note that here, as in general, it is possible that $\beta_k$ achieves several constraints simultaneously; while this has no effect on the algorithm, this phenomenon is the source of some of the technicalities of the proof of Theorem~\ref{theorem.main result 2} related to the index $j(k)$ defined in~\eqref{equation.proof of main result 2.9}.

Having $\beta_4=\frac32$ allows us to define $f_{3;1}$, $f_{3;2}$ and $f_{3;3}$ and moreover compute $\beta_3$ as the largest $t$ such that
\begin{align*}
f_{3;1}(t)&=(t-\tfrac74)^++(t-\tfrac34)^++(t-\tfrac12)^++1\leq\tfrac{15}{4},\\
f_{3;2}(t)&=(t-\tfrac74)^++(t-\tfrac34)^++1\leq\tfrac{7}{4},\\
f_{3;3}(t)&=(t-\tfrac74)^++\tfrac34\leq\tfrac{3}{4},
\end{align*}
namely $\beta_3:=\max\set{(-\infty, \frac{23}{12}]\cap(-\infty, \frac32]\cap(-\infty, \frac74]}=\frac32$.
Since $\set{\beta_3,\beta_4}=\set{\frac32,\frac32}$ we next have
\begin{align*}
f_{2;1}(t)&=(t-\tfrac74)^++(t-\tfrac34)^++1+1\leq\tfrac{15}4,\\
f_{2;2}(t)&=(t-\tfrac74)^++\tfrac34+1\leq\tfrac74,
\end{align*}
and so $\beta_2:=\max\set{(-\infty,\frac{17}8]\cap(-\infty,\frac74]}=\frac74$.
Finally, since $\set{\beta_2,\beta_3,\beta_4}=\set{\frac74,\frac32,\frac32}$,
\begin{align*}
\beta_1:=\max\set{t :f_{1;1}(t)=(t-\tfrac74)^++1+1+1\leq\tfrac{15}4}=\max(-\infty,\tfrac52]=\tfrac52.
\end{align*}

To summarize, in this example the optimal $(\bfalpha,\bfmu)$-completion is the spectrum $\set{\beta_1,\beta_2,\beta_3,\beta_4}=\set{\frac52,\frac74,\frac32,\frac32}$.
Note Theorem~\ref{theorem.main result 2} alone does not tell us how to explicitly construct the completion's corresponding frame vectors,
namely vectors $\set{\bfphi_n}_{n=1}^{5}$ in $\bbF^4$ with $\norm{\bfphi_n}^2=\mu_n$ for all $n$ and such that $\bfA+\sum_{n=1}^{5}\bfphi_n^{}\bfphi_n^*$ has spectrum $\set{\beta_m}_{m=1}^{4}$.
To do that, we can employ the techniques of the previous section, repeatedly applying Lemma~\ref{lemma.Chop Kill step} to take eigensteps backwards from $\set{\beta_m}_{m=1}^{4}$ to $\set{\alpha_m}_{m=1}^{4}$, and then apply the main results of~\cite{CahillFMPS13} to construct $\set{\bfphi_n}_{n=1}^{5}$ from these eigensteps; see~\cite{Poteet12} for examples of this process.
\end{example}

We conclude by briefly discussing a way to implement the algorithm of Theorem~\ref{theorem.main result 2} in general, and in so doing, obtain an upper bound on its computational complexity.
We first compute \smash{$\sum_{n=j}^{N}\mu_n$} for all $j=1,\dotsc,M$.
This can be done using $O(N)$ operations: first find \smash{$\sum_{n=M}^{N}\mu_n$} and then \smash{$\sum_{n=j}^{N}\mu_n=\mu_j+\sum_{n=j+1}^{N}\mu_n$} for all $j=M-1,\dotsc,1$.
Next, for any given $k=M,\dotsc,1$, assume we have already computed $\set{\beta_m}_{m=k+1}^{M}$; we assume nothing in the case where $k=M$.
For any given $j=1,\dotsc,k$, we use at most $O(M)$ operations to compute \smash{$\delta_{k,j}:=\sum_{n=j}^{N}\mu_n-\sum_{m=k+1}^{M}(\beta_m-\alpha_{m-j+1})^+$}.
For this particular $k$ and $j$, we then compute the largest value of $t$ for which \smash{$\sum_{m=j}^{k}(t-\alpha_{m-j+1})^+\leq\delta_{k,j}$}.
A na\"{i}ve implementation of this step involves $O(M^2)$ operations,
yielding $O(M^4+N)$ operations overall.

For a more computationally efficient alternative, note that making the change of variables $l=m-j+1$ gives \smash{$\sum_{m=j}^{k}(t-\alpha_{m-j+1})^+=\sum_{l=1}^{k-j+1}(t-\alpha_l)^+$}.
Indeed, as seen in the previous example, the same piecewise linear functions used in the $k=4$ step reappear in the $k=3,2,1$ steps.
We can exploit this redundancy by performing an out-of-loop computation that evaluates \smash{$g_m(t):=\sum_{l=1}^{m}(t-\alpha_l)^+$} at $t=\alpha_i$ for all $i,m=1,\dotsc,M$.
This has a one-time cost of only $O(M^2)$ operations.
And, returning to our loop, it allows us to quickly find the largest $t$ for which \smash{$g_{k-j+1}(t)=\sum_{m=j}^{k}(t-\alpha_{m-j+1})^+\leq\delta_{k,j}$}.
To be precise, note $g_{k-j+1}$ is nondecreasing, continuous and piecewise linear.
Further note that it only transitions between pieces at points that lie in the nonincreasing sequence $\set{\alpha_i}_{i=1}^{M}$.
As such, taking the smallest index $i$ for which the precomputed value $g_{k-j+1}(\alpha_i)$ is at most $\delta_{k,j}$,
we know the $t$ we seek lies in the interval $[\alpha_i,\alpha_{i-1})$, where $\alpha_0:=\infty$.
Moreover, for $t\in[\alpha_i,\alpha_{i-1})$ the fact that $\set{\alpha_m}_{m=1}^{M}$ is nonincreasing implies $t\geq \alpha_{m-j+1}$ precisely when $m\geq i+j-1$.
Thus, for all $t\in[\alpha_i,\alpha_{i-1})$ we have \smash{$g_{k-j+1}(t)=\sum_{m=j}^{k}(t-\alpha_{m-j+1})^+=\sum_{m=i+j-1}^{k}(t-\alpha_{m-j+1})$}.
In this form, it only takes $O(M)$ operations to find the unique $t\in[\alpha_i,\alpha_{i-1})$ such that $g_{k-j+1}(t)=\delta_{k,j}$.

To summarize, if we are willing to spend $O(M^2)$ operations up front, then for each $k=1,\dotsc,M$ and every $j=1,\dotsc,k$, finding the largest value of $t$ such that \smash{$\sum_{m=j}^{k}(t-\alpha_{m-j+1})^+\leq\delta_{k,j}$} only requires at most $O(M)$ operations.
As such, for each $k=1,\dotsc,M$, finding $\beta_k$ as the minimum of these values of $t$ over all choices of $j=1,\dotsc,k$ requires at most $O(Mk)$ operations.
Summing these over all $k=1,\dotsc,M$, we see an optimal $(\bfalpha,\bfmu)$-completion can be computed in at most $O(M^3+N)$ operations.

\section*{Acknowledgments}
We thank the two anonymous reviewers for their many helpful comments and suggestions.
This work was partially supported by NSF DMS 1042701 and NSF DMS 1321779.
The views expressed in this article are those of the authors and do not reflect the official policy or position of the United States Air Force, Department of Defense, or the U.S.~Government.

\end{document}